\documentclass[11pt]{article}

\usepackage{amsmath,amsthm,amsfonts,amssymb,amscd}
\usepackage[english]{babel}
\usepackage[a4paper]{geometry}
\usepackage{graphicx}
\usepackage{float}
\usepackage{xcolor}
\usepackage[hidelinks]{hyperref}
\usepackage[T1]{fontenc}
\usepackage{lmodern}
\usepackage{mathtools}
\usepackage[hidelinks]{hyperref}

\newcommand{\ee}{\varepsilon}
\newcommand{\pa}{\partial}

\newtheorem{theorem}{Theorem}
\newtheorem{lemma}{Lemma}[section]

\theoremstyle{definition}
\newtheorem{remark}{Remark}[section]
\newtheorem{example}{Example}[section]
\newtheorem{definition}{Definition}[section]

\renewcommand{\hat}{\widehat}
\renewcommand{\tilde}{\widetilde}

\usepackage{enumerate}

\usepackage[]{authblk}

\begin{document}

\date{}

\title{Homogenization of a net of periodic 
	 critically scaled boundary obstacles
	related to reverse osmosis 
	``nano-composite'' membranes}

\author[1]{Jesús Ildefonso D\'iaz\thanks{Corresponding Author: \url{jidiaz@ucm.es}}}

\affil[1]{%
	Instituto de Matem\'{a}tica Interdisciplinar, 
	Universidad Complutense de
	Madrid. \protect\\
	Plaza de Ciencias 3, 28040 (Madrid) Spain.}

\author[1,2]{David Gómez-Castro\thanks{\url{dgcastro@ucm.es}}}
\affil[2]{%
	Dpto.\ de Matem\'atica Aplicada,
	E.T.S. de Ingenier\'ia -- ICAI,
	Universidad Pontificia de Comillas.}

\author[3]{Alexander V. Podolskiy\thanks{\url{originalea@ya.ru}}}
\affil[3]{%
	Faculty of Mechanics and Mathematics, 
	Moscow State University.
	Moscow 19992, Russia.}

\author[3]{Tatiana A. Shaposhnikova\thanks{\url{shaposh.tan@mail.ru}}}

\maketitle
\begin{abstract}
	One of the main goals of this paper is to extend some of the
	mathematical techniques of some previous papers by the authors showing
	that  some very useful phenomenological properties which can be
	observed to the nano-scale can be simulated and justified
	mathematically by means of some homogenization processes when a
	certain critical scale is used in the corresponding framework. Here
	the motivating problem in consideration is formulated in the context
	of the reverse osmosis. We consider, on a part of the boundary of a
	domain $\Omega\subset \mathbb R^n$, a set of very small periodically distributed
	semipermeable membranes having an ideal infinite permeability
	coefficient (which leads to Signorini type boundary conditions) on a
	part $\Gamma_1$ of the boundary. We also assume that a possible chemical
	reaction may take place on the membranes. We obtain the rigorous
	convergence of the problems to a homogenized problem in which there is
	a change in the constitutive nonlinearities. Changes of this type are
	the reason for the big success of the nanocomposite materials. Our
	proof is carried out for membranes not necessarily of radially
	symmetric shape. The definition of the associated critical scale
	depends on the dimension of the space (and it is quite peculiar for
	the special case of $n=2$). Roughly speaking, our result proves that the
	consideration of the critical case of the scale leads to an
	homogenized formulation which is equivalent to have a global
	semipermeable membrane, at the whole part of the boundary $\Gamma_1$, with a
	``finite permeability coefficient of this virtual membrane'' which is
	the best we can get, even if the original problem involves a set of membranes of any arbitrary finite
	permeability coefficients. 
	
	\medskip
	\noindent \textbf{keywords}  homogenization, critical scale, reverse osmosis, Signorini boundary 
	conditions, elliptic partial differential equations, strange term 
\end{abstract} 

\section{Introduction and statement of results}

We present new results concerning the asymptotic behavior, as $\varepsilon
\rightarrow 0,$ of the solution $u_{\varepsilon }$ of \ a family of \
boundary value problems formulated in a cavity (or plant) represented by a
bounded domain $\Omega \subset \mathbb{R}^{n}$, in which a linear diffusion
equation is satisfied. The boundary $\partial \Omega$ is split into two regions. On one the regions, homogeneous Dirichlet conditions are specified. On the other one, some small subsets $G_{\ee}$ is $\ee$-periodically distributed, and
some unilateral boundary condition are specified on them. We also
assume that a possible ``reaction'' may
take place on a net $G_{\varepsilon }$ of small pieces of the boundary given by the periodic repetition of a rescaled particle $G_{0}$.

There are several relevant problems in a wide spectrum of applications
leading to such type of formulations, ranging from water and wastewater
treatment, to food and textile engineering, as well as pharmaceutical and
biotechnology applications (for a recent review see \cite{Mohammad}). One of
them concerns the reverse osmosis when we apply it, for instance, to
desalination processes (see, e.g. \cite{Jamal} and its references). Without
intending to use here a \textquotedblleft realistic model\textquotedblright\ we
shall present an over-simplified formulation that, nonetheless, preserves most of
the mathematical difficulties concerning the passing to the limit as $%
\varepsilon \rightarrow 0$. Some examples of more complex formulations,
covering different aspects of the problems considered here can be found, for
instance, in \cite{GanNeussKnabner} and its many references.\\

We start by recalling that, roughly speaking, semipermeable membranes allow
the passing of certain type of molecules (the so called as
``solvents'') but block other type of molecules (the
``solutes''). The solvents flow from the region of smaller
concentration of solute to the region of higher concentration (the difference of concentration produces the
phenomenon known as \textit{osmotic pressure}). Nevertheless, by creating a
very high pressure it is possible to produce an inverse flow, such as it is
used in desalination plants: it is the so called ``reverse
osmosis''. Since in many cases the semipermeable membrane contains some
chemical products (e.g. polyamides; see \cite{Ghos}), our formulation will
contain also a nonlinear kinetic reaction term in the flux given by a
continuous nondecreasing function $\sigma (s)$. Let us call $w_{\varepsilon
} $ the solvent concentration corresponding to the membrane periodicity
scale $\varepsilon .$ Let us modulate the intensity of the reaction in terms
of a factor $\varepsilon ^{-k},$ with $k$ to be analyzed later. So, for a
critical value of the solvent concentration $\psi $ (associated to the
osmotic pressure) the flux (including the reaction kinetic term) is an
incoming flux with respect to the solvents plant $\Omega $ if the
concentration of the solvent molecules $w(x)$ on the semipermeable membrane $%
G_{\varepsilon }\subset \partial \Omega $ is smaller or equal to this
critical value, but it remains isolated (with no boundary flow, excluding
the reaction term, on the membrane, i.e. when the concentration is $%
w(x)<\psi $)$.$ So, if $\nu $ is the exterior unit normal vector to the
membrane surface we have
\begin{equation*}
\begin{array}{lc}
\text{on }\{x\in G_{\varepsilon }\subset \partial \Omega \text{, }%
w_{\varepsilon }(x)>\psi \} & \partial _{\nu }w_{\varepsilon }+\varepsilon
^{-k}\sigma (\psi -w_{\varepsilon })=0, \\
\text{on }\{x\in G_{\varepsilon }\text{, }w_{\varepsilon }(x)\leq \psi \} &
\partial _{\nu }w_{\varepsilon }+\varepsilon ^{-k}\sigma (\psi
-w_{\varepsilon })=-\varepsilon ^{-k}\mu (\psi -w_{\varepsilon })%
\end{array}%
\end{equation*}%
for some parameter $\mu >0$ called as the \textquotedblleft finite
permeability coefficient of the membrane" (usually, in practice, $\mu $
takes big values). We assume a simplified linear diffusion equation on the
solvent concentration
\begin{equation*}
-\Delta w=F\text{ in }\Omega ,
\end{equation*}%
and some boundary conditions on the rest of the boundary $\partial \Omega $.
For instance, we can distinguish some subregions where Dirichlet or Neumann
types of boundary conditions hold, and so, if we introduce the partition $%
\partial \Omega =\Gamma _{1}\cup \Gamma _{2}$ and assume that, in fact, $%
G_{\varepsilon }\subset \Gamma _{1},$ then we can imagine that
\begin{equation*}
\partial _{\nu }w_{\varepsilon }(x)=h(x)\text{ on }x\in \Gamma _{1}\diagdown
\overline{G_{\varepsilon }},
\end{equation*}%
and
\begin{equation*}
w_{\varepsilon }(x)=g(x)\text{ on }x\in \Gamma _{2}.
\end{equation*}
Figure \ref{fig:reverse osmosis} presents a simplified case of the above mentioned
framework.

\begin{figure}[h]
	\centering 
	\includegraphics[width=.7\textwidth]{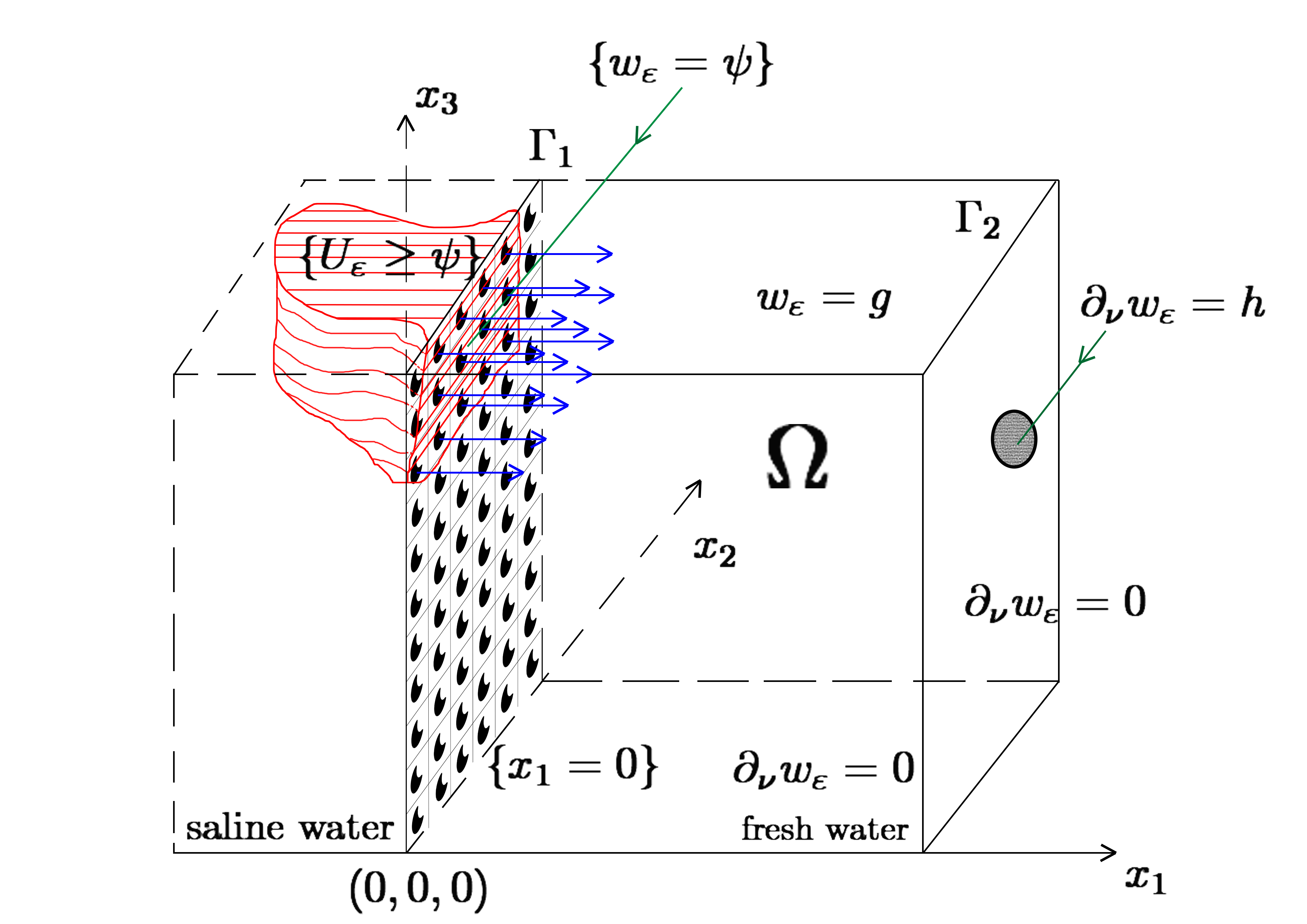}
	\caption{A simple illustration of a plant with a reverse osmosis membrane}
	\label{fig:reverse osmosis}
\end{figure}

We are specially interested in the study of new behaviours arising in the
reverse osmosis membranes having a periodicity $\varepsilon $ of the order
of nanometers (see e.g. \cite{Chae} and its references). Mathematically,
we shall give a sense to those extremely small scales by asking that the
diameter of these subsets included in $G_{\varepsilon }$ let of order $%
a_{\varepsilon },$ where $a_{\varepsilon }\ll \varepsilon $.

Some times, it is interesting to consider semipermeable membranes with an
\textquotedblleft infinite permeability coefficient\textquotedblright\
(formally $\mu =+\infty $, but only for the case $w_{\varepsilon }(x)=\psi $%
) and thus $\psi $ becomes an obstacle which is periodically repeated in $%
G_{\varepsilon }$. Following the approach presented in \cite{Duvaut-Lions},
this can be formulated as%
\begin{equation*}
\begin{array}{l}
\text{on }\{x\in G_{\varepsilon }\text{, }w_{\varepsilon }(x)>\psi \}\text{ }%
\Longrightarrow \text{\ }\partial _{\nu }w_{\varepsilon }+\varepsilon
^{-k}\sigma (\psi -w_{\varepsilon })=0, \\
\text{on }\{x\in G_{\varepsilon }\text{, }%
w_{\varepsilon }(x)=\psi \}\Longrightarrow \partial _{\nu }w_{\varepsilon
}+\varepsilon ^{-k}\sigma (\psi -w_{\varepsilon })\leq 0, \\
\text{and }(w_{\varepsilon }-\psi )(\partial _{\nu }w_{\varepsilon
}+\varepsilon ^{-k}\sigma (\psi -w_{\varepsilon })=0)\text{ on }%
G_{\varepsilon }.%
\end{array}%
\end{equation*}

Now, to carry out our mathematical treatment, it is quite convenient to work
with the new unknown
\begin{equation*}
u_{\varepsilon }(x):=\psi -w_{\varepsilon }(x)
\end{equation*}%
and thus, if we assume (again for simplicity) that $h=g=0$ and $f:=-F$ we
simplify the formulation to arrive at the following formulation which will be the object of study in
this paper:

\begin{equation}
\left\{
\begin{array}{ll}
-\Delta u_{\varepsilon }=f(x), & x\in \Omega , \\
u_{\varepsilon }\geq 0, &  \\
\partial _{\nu }u_{\varepsilon }+\varepsilon ^{-k}\sigma (u_{\varepsilon
})\geq 0, &  \\
u_{\varepsilon }(\partial _{\nu }u_{\varepsilon }+\varepsilon ^{-k}\sigma
(u_{\varepsilon }))=0, & x\in G_{\varepsilon }, \\
\partial _{\nu }u_{\varepsilon }=0, & x\in \Gamma _{1}\setminus \overline{G}%
_{\varepsilon }, \\
u_{\varepsilon }=0, & x\in \Gamma _{2}.%
\end{array}%
\right.  \label{epsilon-problem}
\end{equation}

Notice that in the reaction kinetics we made emerge a re-scaling factor $%
\beta (\varepsilon ):=\varepsilon ^{-k}$ where $k\in \mathbb{R}$. The
relation between the exponent $k$ and the diameter of the chemical particles
(which we shall assume to be given by $a_{\varepsilon }=C_{0}\varepsilon
^{\alpha }$, where $C_{0}>0$ and $\alpha >1$) will be discussed later. This
relation will depend on the dimension of the space $n\geq 3.$ The case $n=2$
is rather special and will require a different treatment: we shall assume
that $a_{\varepsilon }=C_{0}\varepsilon e^{-\frac{\alpha ^{2}}{\varepsilon }%
} $ and $\beta (\varepsilon )=e^{\frac{\alpha ^{2}}{\varepsilon }}$.

Homogenization results for boundary value problems with alternating type of
boundary conditions, including Robin type condition, were widely considered
in the literature. We refer, for instance to the papers \cite{ZuSh07,
Damlamian, Chech93, Borisov} which already contain an extensive bibliography
on the subject. Huge attention was drawn to the similar homogenization
problems but in a domains perforated by the tiny sets on which some
nonlinear Robin type condition is specified on their boundaries. Some
pioneering works in this direction are the papers by Kaizu \cite{Kaizu89,
Kaizu91}. In this works where investigated all the possible relations
between parameters except one the case of the \textquotedblleft
critical\textquotedblright\ relation between parameters $\alpha $ and $\beta
(\varepsilon )$, i.e. $\alpha =k=n/(n-2)$. Later on, this critical case was
considered in \cite{Gon97} for $n=3$ and for the sets $G_{\varepsilon }$
given by balls. It seems that it was in the paper \cite{Gon97} where the
effect of \textquotedblleft nonlinearity change due to the homogenization
process\textquotedblright\ was discovered for the first time. After that, by
using some different method of proof, the critical case was solved for $%
n\geq 3$ in \cite{ZuShDiffEq}. The consideration of the case $n=2$ and for an
arbitrary shape domains $G_{\varepsilon }$ was carried out in \cite{PeShZu14}%
. More recently, many results concerning the asymptotic behavior of
solutions of problems similar to (\ref{epsilon-problem}) were published in
the literature \cite{ZuShDiffEq, JaNeSh, GoLoPePoSh17, GoLoPeSh11,
DiGoPoSh16, DiGoPoSh17, DiGoPoSh Nonlin Anal, DiGoPoSh Math Anal}.
Nevertheless, in all the above mentioned works the particles (or
perforations, according to the physical model used as motivation of the
mathematical formulation) subsets $G_{\varepsilon }$ where assumed to be
balls (having a critical radius). We also mention here the paper \cite
{DiGoShZu Aux Shape} that describes the asymptotic behavior of some related
problem for the case of arbitrary shape sets $G_{\varepsilon }$ and for $n\ge3$%
. One of the main goals of this paper is to extend some of the techniques of
\cite{DiGoShZu Aux Shape} to the problem (\ref{epsilon-problem}) where the
periodically distributed reactions arise merely on some part of the global
boundary $\partial \Omega $ always for the critical scaled case. This is the
case for which some phenomenological properties which arise to the
nano-scale can be simulated and justified by means of homogenization
processes.

\paragraph{Case $n \ge 3$}

In order to present the main results of this paper, and their application to
the reverse osmosis framework, we need to introduce some auxiliary
notations. We start by considering the case $n\geq 3.$ We 
assume that $\Omega $ that it is bounded domain in $\mathbb{R}%
^{n}\cap \{x_{1}>0\}$, $n\geq 3$, with a piecewise-smooth boundary $\partial
\Omega $ that consists of two parts $\Gamma _{1}$ and $\Gamma _{2}$, with
the property that
\begin{equation*}
\Gamma _{1}=\partial \Omega \cap \{x\in \mathbb{R}^{n}:x_{1}=0\}\neq
\emptyset .
\end{equation*}%
We consider a model $G_{0}$ such that $\overline{G_{0}}\subset \{x\in
\mathbb{R}^{n}:x_{1}=0,|x|<1/4\}$ with $\overline{G_{0}}$ diffeomorphic to a
ball. We define $\delta B=\{x:\delta ^{-1}x\in B\}$, $\delta >0$. Let
\begin{equation*}
\widetilde{G_{\varepsilon }}=\bigcup\limits_{j\in \mathbb{Z}^{\prime
}}(a_{\varepsilon }G_{0}+\varepsilon j)=\bigcup\limits_{j\in \mathbb{Z}%
^{\prime }}G_{\varepsilon }^{j},
\end{equation*}%
where
\begin{equation*}
\mathbb{Z}^{\prime }=\{0\}\times \mathbb{Z}^{n-1},
\end{equation*}%
and
\begin{equation}
a_{\varepsilon }=C_{0}\varepsilon ^{k},\qquad k=\frac{n-1}{n-2}\qquad \text{%
and}\qquad C_{0}>0.  \label{assumption a}
\end{equation}%
A justification of the above choice of exponent $k$ can be found, for
instance, in \cite{LoPesuSha 2011} (see also \cite{PeSha}). We define the
net of sets $G_{\varepsilon }$ as the union of sets $G_{\varepsilon
}^{j}\subset \widetilde{G_{\varepsilon }}$ such that $\overline{%
G_{\varepsilon }^{j}}\subset \Gamma _{1}$ and $\rho (\partial \Gamma _{1},%
\overline{G_{\varepsilon }^{j}})\geq 2\varepsilon $, i.e.
\begin{equation*}
G_{\varepsilon }=\bigcup\limits_{{j}\in \Upsilon _{\varepsilon
}}G_{\varepsilon }^{j},
\end{equation*}%
where
\begin{equation*}
\Upsilon _{\varepsilon }=\{j\in \mathbb{Z}^{\prime }:\rho (\partial \Omega ,%
\overline{G_{\varepsilon }^{j}})\geq 2\varepsilon \}.
\end{equation*}%
\begin{figure}[H]
	\centering 
	\includegraphics[width=.4\textwidth]{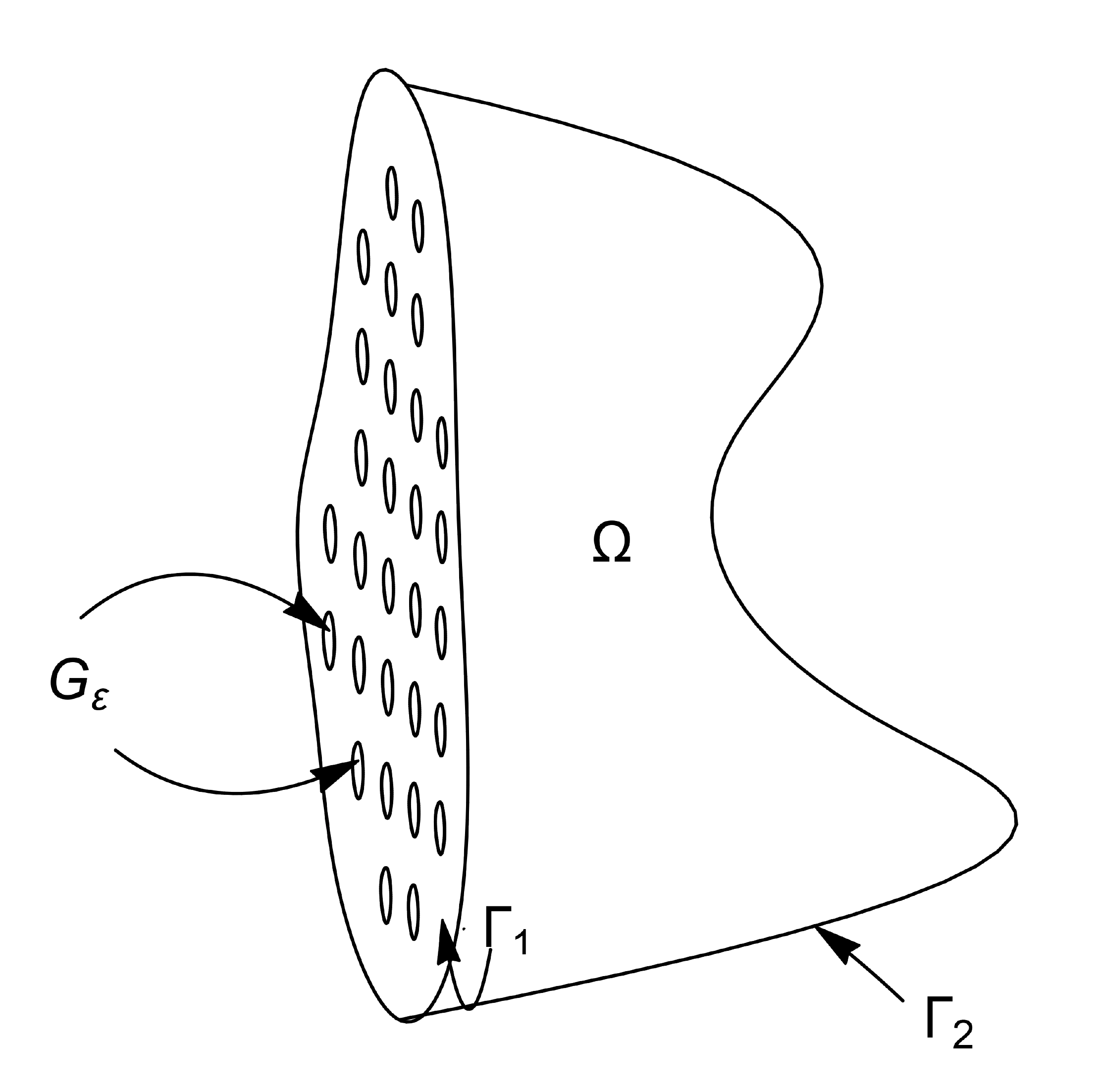}
	\caption{Domain $\Omega$ when $n \ge 3$}
\end{figure}
Notice that $|\Upsilon _{\varepsilon }|\cong d\varepsilon ^{1-n}$, $%
d=const>0 $. It will be useful later to observe that if we denote by $%
T_{r}(x_{0})$ the ball in $\mathbb{R}^{n}$ of radius $r$ centered at a point
$x_{0}$, and if we define the boundary points
\begin{equation*}
P_{\varepsilon }^{j}=\varepsilon j=(0,P_{\varepsilon ,2}^{j},\dots
,P_{\varepsilon ,n}^{j})\qquad \text{ for }j\in \mathbb{Z}^{\prime },
\end{equation*}%
and the set $T_{\varepsilon /4}^{j}=T_{\varepsilon /4}(P_{\varepsilon }^{j})$%
, then we have $\overline{G_{\varepsilon }^{j}}\subset T_{\varepsilon
/4}^{j}.$ 
In this geometrical setting, the so called strong
formulation\textquotedblright\ of the problem for which we want to study the
asymptotic behavior of its solutions is the following:
\begin{equation}
\left\{
\begin{array}{ll}
-\Delta u_{\varepsilon }=f(x), & x\in \Omega , \\
u_{\varepsilon }\geq 0, &  \\
\partial _{\nu }u_{\varepsilon }+\varepsilon ^{-k}\sigma (u_{\varepsilon
})\geq 0, &  \\
u_{\varepsilon }(\partial _{\nu }u_{\varepsilon }+\varepsilon ^{-k}\sigma
(u_{\varepsilon }))=0, & x\in G_{\varepsilon }, \\
\partial _{\nu }u_{\varepsilon }=0, & x\in \Gamma _{1}\setminus \overline{G}%
_{\varepsilon }, \\
u_{\varepsilon }=0, & x\in \Gamma _{2},%
\end{array}%
\right.  \label{original problem}
\end{equation}%
where $\sigma :\mathbb{R}\rightarrow \mathbb{R}$ is a locally H\"{o}lder
continuous non-decreasing function and, at most, super-linear at the
infinity: i.e., such that
\begin{equation}
k_{1}|s-t|\leq |\sigma (t)-\sigma (s)|\leq K_{1}|t-s|^{\rho
_{1}}+K_{2}|s-t|^{\rho _{2}},\text{ for some }\rho _{1},\rho _{2}\in \left(
0,2\right] ,  \label{eq:sigma regularity}
\end{equation}%
for all  $t, s \ge 0$ where $k_1, K_1, K_2 > 0$, $\sigma (0)=0$. In problem
(\ref{original problem}) $\nu =(-1,0,\cdots ,0)$ is the unit outward normal
vector {$\Omega $ at $\{x_{1}=0\}$} and $\partial _{\nu }u={-\frac{\partial u%
}{\partial x_{1}}}$ is the normal derivative of $u$ at this part of the
boundary.
\begin{example}
Notice that examples of such functions cover 
\begin{equation}
	\sigma(s) = \sqrt{s}. 
\end{equation}	
Furthermore, the behaviour at infinity may be superlinear as, for example, in
\begin{equation}
\sigma(s) = \begin{dcases} \sqrt{s} & 0 \le s \le s_0, \\ \sqrt{s_0} +
(s-s_0)^2 & s > s_0. \end{dcases}
\end{equation}
\end{example}

The weak formulation of \eqref{original problem} is the following (see, e.g.
\cite{Duvaut-Lions}): 

\begin{definition}
	We say that $u_\ee$ is a weak solution of \eqref{epsilon-problem} if
	\begin{equation}
	u_{\varepsilon }\in K_{\varepsilon }=\left\{ g\in H^{1}(\Omega ,\Gamma
	_{2}):\,g\geq 0\mbox{ a.e. on }G_{\varepsilon }\right\}  \label{basic set}
	\end{equation}%
	and 
	\begin{equation}
	\int\limits_{\Omega }\nabla u_{\varepsilon }\nabla (\varphi -u_{\varepsilon
	})dx+\varepsilon ^{-k}\int\limits_{G_{\varepsilon }}\sigma (u_{\varepsilon
	})(\varphi -u_{\varepsilon })dx^{\prime }\geq \int\limits_{\Omega }f(\varphi
	-u_{\varepsilon })dx,  \label{integral identity}
	\end{equation}%
	for all $\varphi \in K_{\varepsilon }$. 
\end{definition}
By $H^{1}(\Omega ,\Gamma _{2})$ we
denote the closure in $H^{1}(\Omega )$ of the set of infinitely
differentiable functions in $\overline{\Omega }$, vanishing on the boundary $%
\Gamma _{2}$.\newline

{It is} well known (see, e.g. some references in \cite{DiGoPoSh Nonlin Anal}%
) that problem \eqref{integral identity} has a unique weak solution $%
u_{\varepsilon }\in {K_{\varepsilon }}$. From \eqref{integral identity}, we
immediately deduce that {\
\begin{equation}
\Vert \nabla u_{\varepsilon }\Vert _{L_{2}(\Omega )}\leq K
\label{u epsilon estim}
\end{equation}%
}where, here and below, constant $K$ is independent of $\varepsilon $.
Hence, there exists a subsequence (denoted as the original sequence by $%
\tilde{u}_{\varepsilon }$) such that, as $\varepsilon \rightarrow 0$, we
have
\begin{align}
u_{\varepsilon }& \rightharpoonup u_{0}\mbox{ weakly in }H^{1}(\Omega
,\Gamma _{2}),  \notag \\
u_{\varepsilon }& \rightarrow u_{0}\mbox{ strongly in }L_{2}(\Omega ).
\label{u0 def}
\end{align}

By using the monotonicity of function $\sigma (u)$ one can show (see some
references in \cite{DiGoPoSh Nonlin Anal}) that $u_{\varepsilon }$ satisfies
the following \textquotedblleft very weak formulation\textquotedblright\
\begin{equation}
\int\limits_{\Omega }\nabla \varphi \nabla (\varphi -u_{\varepsilon
})dx+\varepsilon ^{-k}\int\limits_{G_{\varepsilon }}\sigma (\varphi
)(\varphi -u_{\varepsilon })dx^{\prime }\geq \int\limits_{\Omega }f(\varphi
-u_{\varepsilon })dx,  \label{integral inequality}
\end{equation}%
where $\varphi $ is an arbitrary function from $K_{\varepsilon }$.

\bigskip

The main goal of this paper is to consider this critical relation between parameters.
This scale is characterized by the fact that the resulting homogenized problem will contain a so-called ``strange term''
expressing the fact that the character of some nonlinearity arising in the
homogenized problem differs from the original nonlinearity appearing in the
boundary condition of (\ref{original problem}).
Still focusing first on the case $n \ge 3$, it can be shown that critical scale of the size of the holes is given by 
\begin{equation} 
	\alpha = \frac{n-1}{n-2}
\end{equation} (see, e.g., the arguments used
in \cite{LoPesuSha 2011} and \cite{DiGoPoSh Nonlin Anal}). 
The appropriate scaling of the reaction term so that both the diffusive and nonlinear characters are preserved at the limit the limit is, as usual, driven $\ee^{k} \sim |G_\ee|$. Therefore
\begin{equation}
	k = \frac{n-1}{n-2}.
\end{equation}
In the present paper we
construct homogenized problem with a nonlinear Robin type boundary
condition, that contains a new nonlinear term, and prove the corresponding
theorem stating that the solution of the original problem converges as $%
\varepsilon \rightarrow 0$ to the solution of the homogenized problem.\\

We point out that the main difficulties to get an homogenized problem
associated to \eqref{original problem} come from the following different
aspects:
\begin{enumerate} [i)]
	\item  the low differentiability assumed on function $\sigma $ (since
it is non-Lipschitz continuous at $u=0$ and it has quadratic growth at
infinity),
	\item the unilateral formulation of the boundary conditions on $%
G_{\varepsilon }$, 
	\item the general shape assumed on the sets $%
G_{\varepsilon }$, and,
	\item the critical scale of the sets $G_{\varepsilon }^j.$ 
\end{enumerate}
Some of those difficulties where
already in the previous short presentation paper by the authors (\cite
{DiGoPoSh Doklady 2018}) but only for $n=2$, without ii) and by assuming $%
\sigma $ Lipschitz continuous.
  \ Our main goal is to extend our techniques
to the above more general mentioned framework.\\

To build the homogenized problem we still need to introduce some
``capacity type'' auxiliary problems. Given $u\in
\mathbb{R}$, for $y\in (\mathbb{R}^{n})^{+}=\mathbb{R}^{n}\cap \{y_{1}>0\}$%
{,} we introduce the new auxiliary function $\hat{w}(y;u)$, depending also
on $G_{0}$ and $\sigma ,$ as the (unique) solution of the exterior problem
\begin{equation}
\left\{
\begin{array}{ll}
-\Delta \hat{w}=0 & y\in {(\mathbb{R}^{n})^{+}}, \\
\partial _{\nu }\hat{w}-C_{0}\sigma (u-\hat{w})=0 & y\in G_{0}, \\
{\partial _{\nu }\hat{w}=0} & {y\notin G_{0},y_{1}=0,} \\
\hat{w}\rightarrow 0 & \mbox{as }|y|\rightarrow \infty .%
\end{array}%
\right.  \label{hat w prob}
\end{equation}%
Remember that $C_{0}>0$ was given in the structural assumption (\ref%
{assumption a}). The existence and uniqueness of $\hat{w}(y;u)$ is given in
Lemma \ref{existence hat w} below. Let us introduce also the auxiliary
function $\hat{\kappa}(y),y\in {(\mathbb{R}^{n})^{+},}$ as the unique
solution of the problem
\begin{equation}
\begin{dcases}
\Delta \hat{\kappa}=0 & y\in (\mathbb{R}^{n})^{+}, \\
\hat{\kappa}=1 & y\in G_{0}, \\
{\partial _{\nu }\hat{\kappa}=0} & y\notin G_{0},y_{1}=0, \\
\hat{\kappa}\rightarrow 0 & \mbox{as }|y|\rightarrow \infty .%
\end{dcases}  \label{hat kappa prob}
\end{equation}%
We then define the possibly nonlinear function, for $u\in \mathbb{R}$,
\begin{equation}
H_{G_{0}}(u):=\int\limits_{G_{0}}\partial _{\nu }\hat{w}(u,y^{\prime
})dy^{\prime }=C_{0}\int\limits_{G_{0}}\sigma (u-\hat{w}(y^{\prime
};u))dy^{\prime },  \label{H definition}
\end{equation}%
and the scalar
\begin{equation}
\lambda _{G_{0}}:=\int_{G_{0}}\partial _{\nu }\hat{\kappa}(y^{\prime
})\,dy^{\prime },  \label{eq:lambda defn}
\end{equation}%
where in both definitions $y^{\prime }=(0,y_{2},\cdots ,y_{n})$. Notice that
$\hat{\kappa}(y)$ can be extended by symmetry to $\mathbb{R}^{n}\setminus
\overline{G}_{0}$ as an harmonic function. Moreover, by the maximum
principle, $\hat{\kappa}(y)$ reaches its maximum in $G_{0}$, and so by the
strong maximum principle $\partial _{\nu }\hat{\kappa}=-\partial _{x_{1}}%
\hat{\kappa}>0$. Then, we know that
\begin{equation*}
\lambda _{G_{0}}>0.
\end{equation*}%
Some properties of the function $H_{G_{0}}$ will be presented later. In
particular, in Lemma \ref{lem:H Lipschitz} below we will show that $%
H_{G_{0}} $ is always a Lipschitz continuous function (even if $\sigma $ is
merely H\"{o}lder continuous) such that
\begin{equation}
0\leq H_{G_{0}}^{\prime }\leq \lambda _{G_{0}}.  \label{eq:H prime bounds}
\end{equation}%
The following theorem gives a description of the limiting function $u_{0}$
obtained in \eqref{u0 def}.

\begin{theorem}
\label{thm:1} Let $n\geq 3$, $\alpha =k=\frac{n-1}{n-2}$ and $u_{\varepsilon
}$ be a weak solution of the problem \eqref{original problem}. Then function
$u_{0}$ defined in \eqref{u0 def} is a weak solution of the following
problem
\begin{equation}
\left\{
\begin{array}{ll}
-\Delta u_{0}=f, & x\in \Omega , \\
\partial _{\nu }u_{0}+C_{0}^{n-2}H_{G_{0}}(u_{0,+})-\lambda
_{G_{0}}C_{0}^{n-2}u_{0,-}^{{}}=0, & x\in \Gamma _{1}, \\
u_{0}=0, & x\in \Gamma _{2},%
\end{array}%
\right.  \label{homogenized}
\end{equation}%
where $H_{G_{0}}$ is defined by \eqref{H definition} and $\lambda _{G_{0}}$
by \eqref{eq:lambda defn}. Here, as usual $u_{0,+}:=\max \{u_{0},0\}$ and $%
u_{0,-}^{{}}=\max \{-u_{0},0\},$ so that $u_{0}=u_{0,+}-u_{0,-}^{{}}.$%
\end{theorem}

\paragraph{Case $n = 2$}

As mentioned before (see also \cite{DiGoPoSh Doklady 2018}) the case of $n=2$
requires to introduce some slight changes. The domain is given now in the
following way: we consider $\Omega $ be a bounded domain in $\mathbb{R}%
^{2}\cap \{x_{2}>0\}$, the boundary of which consists of two parts $\partial
\Omega =\Gamma _{1}\cup \Gamma _{2}$ and $\Gamma _{1}=\partial \Omega \cap
\{x_{2}=0\}=[-l,l],l>0$, $\Gamma _{2}=\partial \Omega \cap \{x_{2}>0\}$. We
denote
\begin{equation*}
Y_{1}=\{(y_{1},0)|-1/2<y_{1}<1/2\},\quad \hat{l}_{0}=%
\{(y_{1},0)|-l_{0}<y_{1}<l_{0}\}\subset Y_{1},\quad l_{0}\in (0,1/2).
\end{equation*}%
For a small parameter ${\varepsilon }>0$ and $0<a_{\varepsilon }<{%
\varepsilon }$ we introduce the sets

\begin{equation*}
\widetilde{G}_{\varepsilon}=\bigcup_{j\in \mathbb{Z}^{^{\prime
}}}(a_{\varepsilon}\hat{l}_{0}+{\varepsilon}j)=\bigcup_{j\in \mathbb{Z}%
^{^{\prime }}}l^{j}_{\varepsilon},
\end{equation*}
where $\mathbb{Z}^{^{\prime }}$ is a set of vectors $j=(j_{1},0)$ and $j_{1}$
is a whole number. Denote $\Upsilon_{\varepsilon}=\{j\in \mathbb{Z}%
^{^{\prime }}|\overline{l^{j}_{\varepsilon}}\subset \{x=(x_{1},0):x_{1}\in
[-l+2{\varepsilon},l-2{\varepsilon}]\}$. Consider $Y^{j}_{\varepsilon}={%
\varepsilon}Y_{0}+{\varepsilon}j$ and
\begin{equation*}
l_{\varepsilon }=\bigcup_{j\in \Upsilon _{\varepsilon }}l_{\varepsilon }^{j}.
\end{equation*}
It is easy to see that $\overline{l_{\varepsilon }^{j}}\subset
Y_{\varepsilon }^{j}$. Denote $\gamma _{\varepsilon }=\Gamma _{1}\setminus
\overline{l_{\varepsilon }}$. Note that for $\forall j\in \mathbb{Z}%
^{^{\prime }}$, $|l_{\varepsilon }^{j}|=2a_{\varepsilon }l_{0},$ $%
|l_{\varepsilon }|\cong d{a_{\varepsilon }}{\varepsilon }^{-1}.$
\begin{figure}[H]
	\centering 
	\includegraphics[width=.3\textwidth]{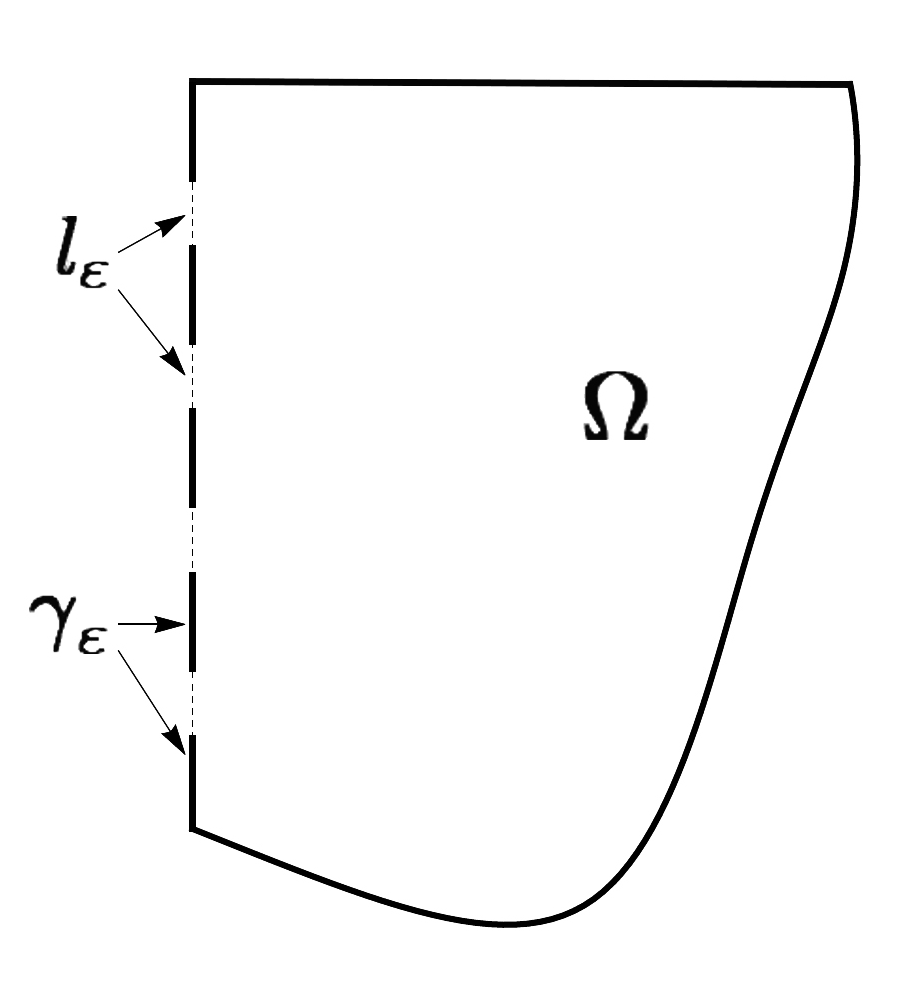}
	\caption{Domain $\Omega$ in two dimensions}
\end{figure}
The formulation of the problem starts by searching
\begin{equation}
u_{\varepsilon }\in K_{\varepsilon }=\left\{ g\in H^{1}(\Omega ,\Gamma
_{2}):\,g\geq 0\mbox{ a.e. on }l_{\varepsilon }\right\}
\label{basic set n=2}
\end{equation}%
be a solution to the following variational inequality
\begin{equation}
\int\limits_{\Omega }\nabla u_{\varepsilon }\nabla (\varphi -u_{\varepsilon
})dx+e^{\frac{\alpha ^{2}}{\varepsilon }}\int\limits_{G_{\varepsilon
}}\sigma (u_{\varepsilon })(\varphi -u_{\varepsilon })dx_{1}\geq
\int\limits_{\Omega }f(\varphi -u_{\varepsilon })dx,
\label{integral identity n=2}
\end{equation}%
where $\varphi $ is an arbitrary function from $K_{\varepsilon }$. This
time, for simplicity, we assume merely that function $\sigma :\mathbb{R}%
\rightarrow \mathbb{R}$ is continuously differentiable, $\sigma (0)=0$ and
and there exists positive constants $k_{1}$, $k_{2}$ such that condition is
satisfied
\begin{equation}
k_{1}\leq \partial _{u}\sigma (u)\leq k_{2},\qquad \forall u\in \mathbb{R}.
\label{sigma cond}
\end{equation}%
(more general terms $\sigma (u)$ where considered in \cite{DiGoPoSh Doklady
2018} but without the Signorini type contraints). Therefore, we have
\begin{equation*}
k_{1}u^{2}\leq u\sigma (u)\leq k_{2}u^{2},\quad \forall u\in \mathbb{R}.
\end{equation*}%
Note that \eqref{basic set n=2}, \eqref{integral identity n=2} is a weak formulation
of the following strong formulation of the problem:
\begin{equation}
\begin{dcases}
-\Delta {u_{\varepsilon }}=f, & x\in \Omega , \\
u_\ee \ge 0 , \\
\quad \partial _{x_{2}}u_{\varepsilon }-\beta (\varepsilon )\sigma (u_{\varepsilon
}) \ge 0\\
\quad u_\ee (\partial _{x_{2}}u_{\varepsilon }-\beta (\varepsilon )\sigma (u_{\varepsilon
}) ) = 0, & x\in l_{\varepsilon }, \\
\partial _{x_{2}}u_{\varepsilon }=0, & x\in \gamma _{\varepsilon }, \\
u_{\varepsilon }=0, & x\in \Gamma _{1},%
\end{dcases}
 \label{problem n=2}
\end{equation}

Our homogenized result in this case is the following:

\begin{theorem}
\label{thm:2} Let $a_{\varepsilon }=C_{0}{\varepsilon }e^{-\frac{\alpha ^{2}%
}{\varepsilon }}$, $\beta (\varepsilon )=e^{\frac{\alpha ^{2}}{\varepsilon }%
} $, $\alpha \neq 0$, $C_{0}>0$ and $u_{\varepsilon }$ be a solution to the
problem \eqref{problem n=2}. Then, there exists a subsequence such that $u_{\varepsilon }\rightarrow u_{0}$, strongly in $L^{2}(\Omega )$, and weakly
in $H^{1}(\Omega ,\Gamma _{1})$, as $\varepsilon \rightarrow 0$, and the
function $u_{0}\in H^{1}(\Omega ,\Gamma _{1})$ is a weak solution to the
following boundary value problem%
\begin{equation}
\left\{
\begin{array}{lr}
-\Delta {u_{0}}=f, & x\in \Omega , \\
\partial _{x_{2}}u_{0}=\frac{\pi }{\alpha ^{2}}%
(H_{l_{0}}(u_{0,+}^{{}})-u_{0,-}^{{}}), & x\in \Gamma _{1}, \\
u_{0}=0, & x\in \Gamma _{2},%
\end{array}%
\right.  \label{homogenized prob n=2}
\end{equation}%
where $H_{l_{0}}(u)$ verifies the functional equation
\begin{equation}
\pi {H}_{l_{0}}(u)=2l_{0}\alpha ^{2}C_{0}\sigma (u-H_{l_{0}}(u)).
\label{func eq}
\end{equation}
\end{theorem}

\bigskip

The special case of dimension $n = 2$ is
illustrative in order to get a complete identification of the function $%
H_{G_{0}}$. Curiously enough, the characterization condition for $H_{G_{0}}$
and ${H}_{l_{0}}$ (see e.g. the functional equations (\ref{H definition})
and (\ref{func eq}), respectively) is quite related (but not exactly the
same) than the one obtained in \cite{DiGoPoSh Nonlin Anal}, although the
problem under consideration in that paper was not the same than problem (\ref%
{original problem}). It was shown in \cite{DiGoPoSh
Nonlin Anal} if $H$: $\mathbb{R}%
\rightarrow \mathbb{R}$ is the solution of a problem
\begin{equation}
	H(u)  = C \sigma (u - H(u))
\end{equation}
then $H$ is given by
\begin{equation}
	H(u) = ( I + (C\sigma)^{-1})^{-1} (u).
\label{defi function H}
\end{equation}%
for $u\in \mathbb{R}$. When the Signorini condition is include, $\sigma$ can be generalized to the maximal monotone graph
\begin{equation}
	\widetilde \sigma (u) = \begin{dcases}
		\sigma(u) & u > 0, \\
		(-\infty,0] & u= 0, \\
		\emptyset & u < 0.
	\end{dcases}
\end{equation}
It was proven there that the corresponding zero order term is given by
\begin{equation}
	\widetilde H (u) = \begin{dcases}
		H(u) & u > 0, \\
		u & u \le 0.
	\end{dcases}
\end{equation}
This behaviour is interesting, because $\widetilde H$ matches \eqref{defi function H} formally. The maximal monotone graph
\begin{equation}
	\gamma (u) = \begin{dcases}
	0 & u > 0, \\
	(-\infty,0] & u= 0, \\
	\emptyset & u < 0.
	\end{dcases}
\end{equation}
has, formally, inverse
\begin{equation}
	\gamma^{-1} (u) = \begin{dcases}
	\emptyset & u > 0,\\
	[0,+\infty) & u = 0, \\
	0 & u< 0.
	\end{dcases}
\end{equation}
In particular
\begin{equation}
	H_\gamma (u):= (I + (C\gamma)^{-1})^{-1} (u) = \begin{dcases}
	\emptyset & u > 0,\\
	[0,+\infty) & u = 0, \\
	u & u< 0.
	\end{dcases}
\end{equation}
In this way, it is clear that
\begin{equation}
	\widetilde H(u) = \begin{dcases}
		H(u) & u > 0, \\
		H_\gamma (u) & u <0.
	\end{dcases}
\end{equation}
For the case $n = 3$ we will make further comments in this direction in Section \ref{sec:mmg}.\\

Coming back to the framework of the semipermeable membranes problems what we
can conclude is that the homogenization of a set of periodic semipermeable
membranes with an \textquotedblleft infinite permeability
coefficient\textquotedblright , in the critical case, leads to an
homogenized formulation which is equivalent to have a global semipermeable
membrane, at $\Gamma _{1}$, with a \textquotedblleft finite permeability
coefficient of this virtual membrane'' $\mu _{_{\infty }\text{ }}$which is
\textit{the best we can get}, even if the original problem involves a finite
permeability $\mu $. Indeed, this comes from the properties of the function $%
H_{G_{0}}$ (which can be also computed for the case of a microscopic
membrane with finite permeability). For instance, on the subpart of the
boundary $\{x\in \Gamma _{1}$, $w_{0}(x)\leq \psi \}$ (with $w_{0}(x):=\psi
-u_{0}(x)$; i.e. where $u_{0}(x)\leq 0$). Moreover, by using that $%
H_{G_{0}}(0)=0$ and the decomposition $u_{0}=u_{0,+}-u_{0,-}^{{}}$, we know
now that $\partial _{\nu }w_{0}=$ $\lambda _{G_{0}}C_{0}^{n-2}w_{0}$ and
thus, for the case of a microscopic finite permeability membrane, it is not
difficult to show that we get an homogenized permeability membrane\
coefficient given by $\lambda _{G_{0}}C_{0}^{n-2},$ which is larger than $%
\mu $ (see more details in Subsection 6.1 below).

\bigskip

The plan of the rest of the paper is the following: Part I (containing
Sections 2-6) is devoted to the study of the case $n\geq 3$ and contains
also some comments and possible extensions (for instance, we give some link
between the present homogenization results and the homogenization of some
problems involving non-local fractional operators). Part II (containing
Section 7) is devoted to the proof of the convergence result for $n=2.$

\bigskip

\part{The case $n \ge 3$}

\section{Estimates on the auxiliary functions}

\subsection{On the auxiliary function $\ \hat{\protect\kappa}$}

{Using the method of sub and supersolutions as in \cite{DiGoShZu Aux Shape}
the following result holds:}

\begin{lemma}
\label{existence hat kappa} There exists a unique solution $\hat{\kappa}\in
\mathbb{X}$ of the problem \eqref{hat w prob}, where
\begin{equation}
\mathbb{X}=\left\{ v\in L_{loc}^{2}(\overline{(\mathbb{R}^{n})^{+}}):\nabla
v\in L^{2}((\mathbb{R}^{n})^{+})^{n},|v|\leq \frac{K}{|y|^{n-2}}\right\}
\end{equation}%
such that
\begin{equation}
0\leq \hat{\kappa}(y)\leq 1\quad a.e.\text{ }y\in {\overline{(\mathbb{R}%
^{n})^{+}},}  \label{hat kappa estim 2}
\end{equation}%
and
\begin{equation}
\hat{\kappa}(y)\leq \frac{K}{|y|^{n-2}}\quad a.e.\text{ }y\in {\overline{(%
\mathbb{R}^{n})^{+}}}.  \label{hat kappa estim}
\end{equation}
\end{lemma}

We define also the family of auxiliary functions {\
\begin{equation}
\hat{\kappa}_{\varepsilon }^{j}=\kappa \left( \frac{x-P_{\varepsilon }^{j}}{%
a_{\varepsilon }}\right) .
\end{equation}%
}

\subsection{On the auxiliary function $\hat w$}

Arguing, as above, in similar terms to the paper {\cite{DiGoShZu Aux Shape}
we get:}

\begin{lemma}
\label{existence hat w} Let $\sigma $ be a maximal monotone graph. There
exists a unique solution $\widehat{w}(\cdot ,u)\in \mathbb{X}$ of problem %
\eqref{hat w prob}. Furthermore, it satisfies that {\ }

\begin{itemize}
\item If $u\geq 0$ then $0\leq \widehat{w}(y,u)\leq u\hat{\kappa}(y)\leq u.$

\item If $u \le 0$ then $0 \ge \widehat w(y, u) \ge u \hat\kappa (y) \ge u$.
\end{itemize}

\noindent Hence
\begin{equation}
|\hat{w}(y;u)|\leq |u|\hat{\kappa}(y)\quad a.e.\text{ }y\in {\overline{(%
\mathbb{R}^{n})^{+}}}.  \label{hat w estim}
\end{equation}
\end{lemma}

Concerning the dependence with respect to the parameter $u$ we start by
considering the case in which $\sigma $ is a maximal monotone graph
associated to a continuous function:

\begin{lemma}
Let $\sigma$ be a nondecreasing continuous function such that $\sigma(0) = 0$%
. Then
\begin{equation}
| \hat w (u_1, y) - \hat w(u_2, y) | \le |u_1 - u_2|.
\end{equation}
for all $u_1, u_2 \in \mathbb{R}$.
\end{lemma}

\begin{proof}
    Let $\hat w (u_1, y),\, \hat w(u_2, y)$ be two solutions of the problem \eqref{hat w prob} with parameters $u_1,\, u_2 \in\mathbb{R}$. Consider the function $v = w(u_1, y) - w(u_2, y)$. This function is a solution of the following exterior problem
    \begin{equation}\label{Problem for difference of solutions}
    \left\{
    \begin{array}{ll}
    \Delta v = 0, & y \in(\mathbb R ^n )^+,\\
    \pa_{\nu_y} v = C_0\left(\sigma(u_1 - \hat w(u_1,y)) - \sigma(u_2 - \hat w(u_2,y))\right), & y \in G_0,\\
    v \rightarrow 0 & \mbox{as } |y| \rightarrow \infty,
    \end{array}
    \right.
    \end{equation}

    First consider the case $u_1 > u_2$. If we choose $v^-$ as a test function in the integral identity for the problem \eqref{Problem for difference of solutions} we arrive at
    \begin{equation*}
    \int\limits_{(\mathbb R ^n )^+}|\nabla v^-|^2dx + C_0\int\limits_{G_0}\left(\sigma(u_1 - \hat w(u_1,y)) - \sigma(u_2 - \hat w(u_2,y))\right)v^- ds = 0
    \end{equation*}
    The second integral in the obtained expression can be nonzero only if $v < 0$, i.e. $\hat w(u_1, y) - \hat w(u_2, y) < 0$. By combining this inequality with the condition $u_1 > u_2$ we get $u_1 - \hat w(u_1, y) > u_2  - \hat w(u_2, y)$. This inequality and monotonicity of the function $\sigma$ imply that second integral is non-negative. Hence, two integrals must be equal to zero, so $v^- = 0$ $\mathcal{L}^{n-1}$-a.e. in $G_0$ and $v^- = c$ in $(\mathbb{R}^n)^+$. But we have $v\rightarrow 0$ as $|y|\rightarrow\infty$, hence, $c = 0$, i.e. $\hat w(u_1, y) - \hat w(u_2, y) \ge 0$.

    One can construct function $\varphi(r) \in C^{\infty}_0(\mathbb{R})$ such that $\varphi = 0$ if $|r| > 1$ and $\varphi = 1$ if $|r| < 0.5$. We take $(u_1 - u_2 - v)^{-}\varphi(\rho(x, G_0)/R)$ as a test function in the integral identity for the problem \eqref{Problem for difference of solutions} and obtain
    \begin{align*}
    -\int\limits_{(\mathbb R ^n )^+}\nabla v & \nabla (u_1 - u_2 - v)^{-}\varphi(\rho(x, G_0)/R)dx\,-\\
    & \quad -\int\limits_{(\mathbb R ^n )^+}\frac{\varphi'(\rho(x, G_0)/R)}{R}(u_1 - u_2 - v)^{-}\nabla v\nabla\rho dx\,+\\
    & \quad +C_0\int\limits_{G_0} \left(\sigma(u_1 - \hat w(u_1,y)) - \sigma(u_2 - \hat w(u_2,y))\right)(u_1 - u_2 - v)^{-}ds  \,\\
    &  = I_{1, R} + I_{2, R} + I_3 = 0.
    \end{align*}
    Since $\sigma$ is monotone $I_3 \le 0$.
	\\
    For the first integral we have
    \begin{equation*}
    I_{1, R} \le -\int\limits_{\mathcal{D}_{1,R}} |\nabla(u_1 - u_2 - v)^-|^2 dx \le 0,
    \end{equation*}
    where $\mathcal{D}_{1,R} = ((\mathbb R ^n )^+)\bigcap \{x\in\mathbb{R}^n: \rho(x, G_0) < R\}$. We have that
    \begin{equation*}
    I_1 = -\int\limits_{(\mathbb R ^n )^+}|\nabla(u_1 - u_2 - v)^-|^2 dx = \lim\limits_{R\rightarrow\infty}I_{1,R}\le 0
    \end{equation*}
    For the second integral we derive estimation
    \begin{align*}
    |I_{2, R}| &\le K_1\int\limits_{\mathcal{D}_{2,R}}\frac{|v|}{R}|\nabla v| dx \le \frac{K_1}{R}\Vert v\Vert_{L_2(\mathcal{D}_{2,R})}\Vert\nabla v\Vert_{L_2((\mathbb{R}^n)^+)}\\
    &\le \frac{K_2}{R^{\frac{n-2}{2}}}\Vert\nabla v\Vert_{L_2((\mathbb{R}^n)^+)} \rightarrow 0,\quad\mbox{as } R\rightarrow\infty
    \end{align*}
    where $\mathcal{D}_{2,R} = ((\mathbb R ^n )^+)\bigcap \{x\in\mathbb{R}^n: R/2 < \rho(x, G_0) < R\}$. Thereby, as $R\rightarrow\infty$ we have
    $
    I_1 + I_3 = 0, I_1 \le 0, I_3 \le 0,
    $
    and so $$I_1 = 0, \quad I_3 = 0.$$
    Taking into account that $v \rightarrow 0 \mbox{ as } |y| \rightarrow \infty$ we derive from the last corollary that $(u_1 - u_2 - v)^{-} \equiv 0$, i.e. $0 \le v < u_1 - u_2$ in $(\mathbb{R}^n)^+$. Moreover, we have that $v < u_1 - u_2$\, $\mathcal{L}^{n-1}$-a.e. in $G_0$.

    The case $u_1 < u_2$ is analogous to that above, so we have $u_1 - u_2 \le v \le 0$ in $(\mathbb{R}^n)^+$ and $\mathcal{L}^{n-1}$-a.e. in $G_0$. This concludes the proof.
\end{proof}

The use of the comparison principle leads to an additional conclusion:

\begin{lemma}
Let $u_1 > u_2$ then
\begin{equation}  \label{eq:comparison u1 u2 and kappa}
0 \le \hat w (u_1 , y) - \hat w(u_2, y) \le (u_1- u_2 ) \hat \kappa (y).
\end{equation}
\end{lemma}

\begin{proof}
    The functions $\varphi_1 (y) = \hat w (u_1 , y) - \hat w(u_2, y)$ and $\varphi_2  (y)= (u_1 - u_2) \hat \kappa(y)$ can be extended (by symmetry) as harmonic functions to $\mathbb R^n \setminus \overline{G_0}$ such that $\lim_{|y|\to \infty} (\varphi_2 - \varphi_1 ) = 0$ and $\varphi_2 - \varphi_1 \ge 0$ on $G_0$. The comparison principles proves the result.
\end{proof}

A more regular dependence with respect to $u$ can be also proved under
additional regularity on function $\sigma $:

\begin{lemma}[Differentiable dependence of solutions]
Let $\sigma \in \mathcal{C}^{1}$ and $\sigma ^{\prime }\geq k_{1}>0$. Then,
the map $u\in \mathbb{R}\mapsto \hat{w}(u,\cdot )\in L_{loc}^{2}(K)$ is
differentiable, for every smooth bounded set $K$ such that $G_{0}\subset
K\subset \overline{(\mathbb{R}^{n})^{+}}$. Furthermore, if we define
\begin{equation}
\hat{W}(u,y)=\frac{\partial \hat{w}(y;u)}{\partial u}
\end{equation}%
then
\begin{equation}
\int\limits_{(\mathbb{R}^{n})^{+}}\nabla \hat{W}(u,y)\nabla \varphi
dy=C_{0}\int\limits_{G_{0}}\sigma ^{\prime }(u-\widehat{w}(u,y))\left( 1-%
\hat{W}(u,y)\right) \varphi (y)dS_{y}  \label{eq:weak formulation of W}
\end{equation}%
for $\varphi \in C_{c}^{\infty }(\overline{(\mathbb{R}^{n})^{+}})$, $\nabla
\hat{W}\in L^{2}(\Omega )^{n}$ and $0\leq \hat{W}(u,y)\leq \hat{\kappa}(y)$.
In particular,
\begin{equation}
0\leq \frac{\partial \widehat{w}}{\partial u}\leq \hat \kappa (y).
\end{equation}
\end{lemma}

\begin{proof}
    Considering the difference of two solutions
    \begin{align*}
    \int \limits _ {(\mathbb R^n)^+} \nabla \frac{ \hat w (u + h , y) - \hat w (u,y) }{h} \nabla \varphi dy &= C_0 \int \limits _{ G_0 } \frac{ \sigma (u + h - \hat w (u+h,y) ) -  \sigma (u - \hat w (u,y) ) }{h} \varphi dS_y \\
    &= C_ 0\int_{ G_0 } \sigma' (\xi_h (y)) \left (1 - \frac{ \hat w (u + h , y) - \hat w (u , y)   }{h}  \right) \varphi(y) dS_y
    \end{align*}
    for some $\xi_h$ in between $u + h - \hat w (u + h , y)$ and $u - \hat w (u , y) $. From this
    \begin{align} \label{eq:weak formulation u plus h minus u}
    \int \limits _ {(\mathbb R^n)^+} \nabla\frac{ \hat w (u + h , y) - \hat w (u,y) }{h} \nabla\varphi dy& + C_0 \int_{ G_0 } \sigma' (\xi_h (y)) \frac{ \hat w (u + h , y) - \hat w (u , y)   }{h} \varphi(y) dS_y  \nonumber \\
    &= C_0 \int_{G_0} \sigma' (\xi_h (y)) \varphi (y) dS_y.
    \end{align}
    Taking $\varphi = \frac{ \hat w (u + h , y) - \hat w (u,y) }{h} $, and using the fact that $\hat w(u, y)$ can be bounded and $\sigma'$ is continuous
    \begin{equation}
    \left\|  \nabla  \left( \frac{ \hat w (u + h , y) - \hat w (u,y) }{h} \right)  \right\|_{L^2 (\Omega)^n}^2  + k_1 C_0 \left\|  \frac{ \hat w (u + h , y) - \hat w (u,y) }{h}   \right\|_{L^2 (G_0)}^2 \le C.
    \end{equation}
    for $h$ small. Thus, $ \frac{ \hat w (u + h , y) - \hat w (u,y) }{h}$ admits a weak limit as $h\to 0$ in $H^1 (K)$, let it be $\hat W (u,y)$. Thus, up to a subsequence, it admits a pointwise limit and strong limit in $L^2 (K)$. It is clear that
    $$\xi_h (y) \to u - \hat w (u,y), \qquad \textrm{ pointwise as } h \to 0.$$
    By passing to the limit for $\varphi \in \mathcal C_c^\infty (\overline{(\mathbb R^n)^+})$ fixed, we characterize \eqref{eq:weak formulation of W}. From \eqref{eq:comparison u1 u2 and kappa} we deduce that, for $h > 0$
    \begin{equation*}
    0 \le \frac{ \hat w(u+h,y) - \hat w (u, y)}{ h } \le \kappa(y).
    \end{equation*}
    As $h \to 0$ we deduce the result $0 \le \hat W (u,y) \le \kappa(y)$.
\end{proof}

	\begin{remark}
		Notice that $\widehat W$ is the unique solution of
		\begin{equation}
			\begin{dcases}
				-\Delta \widehat W = 0 & y \in (\mathbb R^n)^+, \\
				\partial_\nu \widehat W + C_0 \sigma'(u - \widehat w) \widehat W = C_0 \sigma' (u - \widehat w) u & y \in G_0, \\
				\partial _\nu \widehat W = 0 & y \notin G_0, y_1 = 0, \\
				\widehat W \to 0 & |y| \to +\infty.
			\end{dcases}			
		\end{equation}
	\end{remark}

	\begin{remark}
		\label{rem:linear sigma implies linear w and H}
		Assume $\sigma (u) = \mu u$. Then $\sigma'(u - \widehat w) \equiv \mu$, and $\widehat W$ does not depend on $\mu$. Therefore, $\widehat w(x,u) = u\widehat W(x)$. Furthermore
		\begin{equation}
			H(u) = \lambda_\mu u
		\end{equation}
	\end{remark}

\subsection{On the regularity of function $H$}

\begin{lemma}
\label{lem:H Lipschitz} Let $\sigma $ be a maximal monotone graph. The
function $H_{G_{0}}(u)$ defined by \eqref{H definition} is Lipschitz
continuous nondecreasing of constant $\lambda _{G_{0}}$ given by %
\eqref{eq:lambda defn}, i.e. if $u_{1}>u_{2}$
\begin{equation}
0\leq H(u_{1})-H(u_{2})\leq \lambda _{G_{0}}(u_{1}-u_{2}),
\label{eq:Lipschitz H}
\end{equation}%
i.e., in the notation of weak derivatives
\begin{equation}
0\leq H^{\prime }(u)\leq \lambda _{G_{0}}\text{ \ for a.e. }u\in \mathbb{R}.
\end{equation}
\end{lemma}

\begin{remark}
Notice that $\hat{\kappa}$ (and thus $\lambda _{G_{0}}$) does not depend on $%
\sigma $, but only on $G_{0}$.
\end{remark}

\begin{proof}
        First, let $\sigma$ be smooth and $\sigma' \ge k_1> 0$. Again, let $\hat W = \frac{\partial \hat w}{\partial u}$. Taking derivatives in \eqref{H definition} we have that
        \begin{equation}
        H'(u) = C_0 \int_{G_0} \sigma' (u - \widehat w) (1 - \hat W)
        dy'.
        \end{equation}
        Since $\hat W \le 1$ we have that $H' \ge 0$.
        Using $\hat \kappa$ as a test function in \eqref{eq:weak formulation of W}  we obtain that
        \begin{align*}
        H'(u) &= C_0 \int_{G_0} \sigma'(u - \hat w ) ( 1 - \hat W) dy' = \int_{G_0} \Big ( \sigma'(u - \hat w ) ( 1 - \hat W) \Big) \hat \kappa dy'\\
        & = \int_{ (\mathbb R^n)^+ } \nabla \hat W \nabla \hat \kappa dy= \int_{G_0} \hat W (\partial_\nu \hat \kappa) dy' \le \int_{G_0} (\partial _\nu \hat \kappa ) dy' = \lambda
        \end{align*}
        using the facts that $\partial_\nu \hat \kappa \ge 0$ and $0 \le \widehat W \le 1$.
        \newline
        If $\sigma$ is a general maximal monotone graph, estimate \eqref{eq:Lipschitz H} is maintained by approximation by a smooth sequence of function $\sigma_k$.
    \end{proof}

\section{Convergence of the boundary integrals where $u_\protect\varepsilon %
\ge 0$}

\subsection{On the auxiliary function $w_\protect\varepsilon^j$}

We introduce a function $w^j_\varepsilon(u, x)$ as a solution of the
boundary value problem
\begin{equation}  \label{w prob}
\left\{
\begin{array}{ll}
\Delta w^j_\varepsilon = 0, & x \in (T^{j}_{\varepsilon/4})^+, \\
\partial_\nu w^j_\varepsilon = \varepsilon^{-k}\sigma(u - w^j_\varepsilon),
& x \in G^j_\varepsilon, \\
{\ \partial_\nu w^j_\varepsilon = 0 } & {\ x \in ( T_{\varepsilon/4}^j)^0
\setminus \overline { G^j _\varepsilon }, } \\
w^j_\varepsilon = 0, & x \in {\ (\partial T^j_{\varepsilon/4})^+ },%
\end{array}
\right.
\end{equation}
where $u\in\mathbb{R}$ is a parameter. {\ We will compare this auxiliary
functions with the functions
\begin{equation}
\hat w_\varepsilon^j (u,x)= \hat w\left (u, \frac{x - P^j_\varepsilon}{%
a_\varepsilon} \right)
\end{equation}
}

The function $w^j_\varepsilon\in H^1_0(T^j_{\varepsilon/4})$ is a weak
solution of the problem \eqref{w prob} if it satisfies the integral identity
\begin{equation*}
\int\limits_{{\ (T^{j}_{\varepsilon/4})^+ }}\nabla w^j_\varepsilon \nabla\varphi dx -
\varepsilon^{-k}\int\limits_{G^j_\varepsilon}\sigma(u -
w^j_\varepsilon)\varphi dx^{\prime }= 0,
\end{equation*}
for an arbitrary function $\varphi\in H^1_0(T^j_{\varepsilon/4})$.

{From the uniqueness of problem \eqref{w prob} and the method of sub and
supersolutions, we have the following:}

\begin{lemma}
\label{aux fun w estim parameter} The function $w^j_\varepsilon$ satisfies
following estimations

\begin{enumerate}
\item if $u \ge 0$ then $0 \le w^j_\varepsilon(u, x) \le {\hat
w_\varepsilon^j (u,x) \le u}$,

\item if $u \le 0$ then $u \le {\hat w_\varepsilon^j (u,x) \le }
w^j_\varepsilon(u, x) \le 0$.
\end{enumerate}
\end{lemma}

{\ }

\begin{remark}
Since $\sigma$ is monotone, from the previous result,
\begin{equation}  \label{eq:estimate sigma u - w}
|\sigma(u - w_\varepsilon^j) | \le |\sigma(u)|
\end{equation}
\end{remark}

We define the function
\begin{equation}  \label{W def}
W_\varepsilon = \begin{dcases} w^j_\ee(u,x), & x\in {(T^{j}_{\ee/4})^+},\,
j\in \Upsilon_\ee\\ 0, &
x\in{(\mathbb{R}^n)^+\setminus\bigcup\limits_{j\in\Upsilon_\ee}%
\overline{(T^j_{\ee/4})^+}}. \end{dcases}
\end{equation}
Notice that $W_\varepsilon \in H^1(\Omega, \Gamma_2)$ for all $u\in \mathbb{R%
}$. The following Lemma proposes estimate of the introduced function and its
gradient

\begin{lemma}
\label{lemma W estim} The following estimations for the function $%
W_\varepsilon$, that was defined in \eqref{W def}, are valid
\begin{align}
\Vert \nabla W_\varepsilon \Vert^2_{L_2(\Omega)} &\le K {|u||\sigma(u)|},
\notag \\
\Vert W_\varepsilon\Vert^2_{L_2(\Omega)} &\le K\varepsilon^2{|u||\sigma(u)|},
\label{W estim}
\end{align}
\end{lemma}

\begin{proof}
    From the weak formulation of problem \eqref{w prob} for
    $j\in\Upsilon_\ee$ we know
    \begin{equation*}
    \int\limits_{{(T^j_{\ee/4})^+}}\nabla w^j_\ee \nabla\varphi dx - \ee^{-k}\int\limits_{G^j_\ee}\sigma(u - w^j_\ee)\varphi dx'  = 0.
    \end{equation*}
    We take $\varphi = w^j_\ee$ as a test function in this expression and obtain
    \begin{equation*}
    \int\limits_{{(T^j_{\ee/4})^+}}|\nabla w^j_\ee|^2 dx - \ee^{-k}\int\limits_{G^j_\ee}\sigma(u - w^j_\ee)w^j_\ee dx'  = 0.
    \end{equation*}
    Then we can transform the obtained relation to the following
    expression
    \begin{equation*}
    \int\limits_{{(T^j_{\ee/4})^+}}|\nabla w^j_\ee|^2 dx + \ee^{-k}\int\limits_{G^j_\ee}\sigma(u - w^j_\ee)(u - w^j_\ee) dx'  = \ee^{-k}\int\limits_{G^j_\ee}\sigma(u - w^j_\ee)u
    dx'.
    \end{equation*}
    By using the monotonicity of $\sigma(u)$ we derive the following inequality
    \begin{gather*}
    \int\limits_{{(T^j_{\ee/4})^+}}|\nabla w^j_\ee|^2 dx + \ee^{-k}\int\limits_{G^j_\ee}\sigma(u - w^j_\ee)(u - w^j_\ee) dx'  \le \ee^{-k}\int\limits_{G^j_\ee}|\sigma(u - w^j_\ee)||u|
    dx'.
    \end{gather*}
    {
        Due to the monotonicity $\sigma$ and \eqref{eq:estimate sigma u - w} we have that
        \begin{gather*}
        \Vert\nabla w^j_\ee\Vert^2_{L_2({(T^j_{\ee/4})^+})}  \le \ee^{-k}\int\limits_{G^j_\ee}|\sigma(u - w^j_\ee)||u| dx'
        \le { k_2 |u| |\sigma(u)| \ee^{-k}|G^j_\ee| }
        \end{gather*}}
    Hence, the following estimate is valid
    \begin{equation*}
    \Vert\nabla w^j_\ee\Vert^2_{L_2({(T^j_{\ee/4})^+})} \le K |u| { |\sigma(u)| }
    \ee^{n-1}.
    \end{equation*}
    Adding over all cells we get
    \begin{equation*}
    \Vert\nabla W_\ee\Vert^2_{L_2(\Omega)}  \le K {|u||\sigma(u)|}.
    \end{equation*}
    Friedrich's inequality implies
    \begin{equation*}
    \Vert w^j_\ee\Vert^2_{L_2({(T^j_{\ee/4})^+})} \le K\ee^2\Vert\nabla
    w^j_\ee\Vert^2_{L_2({(T^j_{\ee/4})^+})}.
    \end{equation*}
    Summing over all cells and using obtained estimations we derive
    \begin{equation*}
    \Vert W_\ee\Vert^2_{L_2(\Omega)} \le K\ee^2\Vert\nabla W_\ee\Vert^2_{L_2(\Omega)} \le K \ee^2
    {|u||\sigma(u)|},
    \end{equation*}
    which concludes the proof.
\end{proof}

Hence, as $\varepsilon\rightarrow0$ we have
\begin{equation}  \label{W conv}
\begin{split}
W_\varepsilon \rightharpoonup 0 \mbox{ weakly in } H^1(\Omega), \\
W_\varepsilon \rightarrow 0 \mbox{ strongly in } L_2(\Omega).
\end{split}%
\end{equation}

\subsection{The comparison between $w_\protect\varepsilon^j$ and $\hat w_%
\protect\varepsilon^j$}

{As an immediate consequence of Lemma \ref{aux fun w estim parameter} we
have that}

\begin{lemma}
\label{lemma estim of w by hat w} For all $u\in\mathbb{R}$ and a.e. $x\in {%
(T^j_{\varepsilon/4})^+}$ we have
\begin{equation}  \label{estim of w by hat w}
|w^j_\varepsilon(u, x)| \le \left| {\hat w _\varepsilon^ j (u,x)} \right|
\end{equation}
\end{lemma}

The following Lemma gives an estimate of proximity of functions $%
w^j_\varepsilon$ and $\hat w$.

\begin{lemma}
\label{lemma w and hat w} For the introduced functions $w_{\varepsilon
}^{j}(u,x)$ and $\hat{w}(y;u)$ following estimations hold
\begin{align}
\Vert \nabla (v_{\varepsilon }^{j}(u,x))\Vert _{L_{2}({(T_{\varepsilon
/4}^{j})^{+}})}^{2}& \leq K{|u|^{2}}\varepsilon ^{n},
\label{w and hat w estim 2} \\
\Vert v_{\varepsilon }^{j}(u,x)\Vert _{L_{2}({(T_{\varepsilon /4}^{j})^{+}}%
)}^{2}& \leq K{|u|^{2}}\varepsilon ^{n+2},
\end{align}%
where $v_{\varepsilon }^{j}(u,x)=w_{\varepsilon }^{j}(u,x)-{\hat{w}%
_{\varepsilon }^{j}(u,x)}$.
\end{lemma}

\begin{proof}
The function $v^j_\ee$ is solution to the following boundary value
problem
\begin{equation*}
\left\{
\begin{array}{ll}
\Delta v^j_\ee = 0, & x\in {(T^j_{\ee/4})^+} \setminus\overline{G^j_\ee},\\
\pa_{\nu}v^j_\ee = \ee^{-k}\left(\sigma(u - {\hat w_\ee^j}) - \sigma(u - w^j_\ee) \right), & x\in G^j_\ee,\\
{ \pa_\nu v^j_\ee = 0 } & { x \in ( T_{\ee/4}^j)^0 \setminus \overline { G^j _\ee }, }\\
v^j_\ee = - {\hat w_\ee^j}, & x\in {(\pa T^j_{\ee/4})^+}.
\end{array}
\right.
\end{equation*}
Applying the comparison principle we have $|v^j_\ee (u,x)|\le |{ \hat
w_\ee^j (u, x)}|$ a.e. in ${(T^j_{\ee/4})^+}$. We take $v^j_\ee$ as
a test function in the corresponding weak solution integral
expression for the above problem
\begin{equation} \label{eq:first estimate vee}
{\Vert\nabla v^j_\ee \Vert_{L_2({(T^j_{\ee/4})^+})}^2} +
\ee^{-k}\int\limits_{G^j_\ee}\left(\sigma(u - w^j_\ee) - \sigma(u -
{\hat w_\ee^j}) \right)v^j_\ee dx' = - \int\limits_{{(\pa
T^j_{\ee/4})^+}} \pa_{\nu}v^j_\ee {\hat w_\ee^j} ds.
\end{equation}
We transform the right-hand side expression of the inequality in the
following way:
\begin{equation*}
- \int\limits_{ {( \pa T^j_{\ee/4} )^+}} \pa_{\nu}v^j_\ee {\hat w_\ee^j} ds = -\int\limits_{  {(T^j_{\ee/4})^+ \setminus \overline{ (T^{j}_{\ee/8} )^+}} }\nabla v^j_\ee\nabla{\hat w_\ee^j} dx + \int\limits_{{(\pa T^{j}_{\ee/8})^+}}\pa_\nu v^j_\ee {\hat w_\ee^j} ds.
\end{equation*}
Let us estimate the obtained terms. {We can extend $v_\ee^j$ by
symmetry $v_\ee^j (x,u) = v^\ee_j (-x,u)$ for $x \in (T_{\ee/4}^j
)^- $ which is harmonic in $T_{\ee / 4}^j \setminus \overline
{G_0}$.} By using some estimates on the derivatives of harmonic
functions and the maximum principle, for $\tilde x \in \pa
T^{j}_{\ee/8}$, we get
\begin{gather*}
|\pa_{x_i}v^j_\ee(\tilde x)| \le \frac{1}{|T_{\ee/16}(\tilde x)|}\bigg|\int\limits_{T_{\ee/16}(\tilde x)}\frac{\pa v^j_\ee}{\pa x_i} dx\bigg| = \\
= \frac{K}{\ee^n}\bigg|\int\limits_{\pa T_{\ee/16}(\tilde x)}v^j_\ee\nu_i dx\bigg| \le K{|u|},
\end{gather*}
since, on $\pa T_{\ee/16}(\tilde x)$, from the maximum principle we
have
\begin{equation} \label{eq:pointwise estimate vee}
|v^j_\ee|\le|{\hat w_\ee^j}|\le\frac{K{|u|}}{\left|(x-P^j_\ee)/a_\ee\right|^{n-2}}=K{|u|}\ee^{2-n}a_\ee^{n-2}=K{|u|}\ee^{2-n}\ee^{n-1}=K{|u|}\ee.
\end{equation}
This last estimate implies that $|\nabla v^j_\ee(\tilde x)|\le K$
for $\tilde x \in \pa T^j_{\ee/8}$. Therefore, we get an estimate of
the second term
\begin{equation*}
\bigg|\int\limits_{{ (\pa T^{j}_{\ee/8})^+ } }\pa_\nu v^j_\ee {\hat
w_\ee^j} ds\bigg| \le K {|u|}\max\limits_{\pa T^{j}_{\ee/8}}|{\hat
w_\ee^j}||\pa T^{j}_{\ee/8}|\le K{|u|^2}\ee^{n.}
\end{equation*}
Then we estimate the first term
\begin{equation*}
\bigg|\int\limits_{  {(T^j_{\ee/4})^+ \setminus \overline{ (T^{j}_{\ee/8} )^+}} }\nabla v^j_\ee\nabla\hat w dx \bigg|\le K{|u|^2}\ee a^{n-2}_\ee = K {|u|^2} \ee^n
\end{equation*}
Combining the obtained estimates and using the properties of
function $\sigma$ we get
\begin{equation*}
\Vert\nabla v^j_\ee\Vert^2_{L_2({(T^j_{\ee/4})^+})} \le K { |u|^2 }\ee^{n}
\end{equation*}
Friedrichs' inequality implies that
\begin{equation*}
\Vert v^j_\ee\Vert^2_{L_2({(T^j_{\ee/4})^+})}\le K {|u|^2}\ee^{n+2}
\end{equation*}
This concludes the proof.
\end{proof}
{\color{red} }

\begin{lemma}
\label{lem:trace theorem vee} We have that
\begin{equation}
\frac{1}{|G_{\varepsilon }^{j}|}\int_{G_{\varepsilon }^{j}}\left\vert
v_{\varepsilon }^{j}(u,x^{\prime })\right\vert ^{2}dx^{\prime }\leq
\varepsilon .
\end{equation}
\footnote{NOTE: Function $v_\ee^j$ in this section depends on $u$.}
\end{lemma}

\begin{proof}
    From \eqref{eq:first estimate vee}, using \eqref{eq:sigma regularity}, we deduce that
    \begin{equation}
        k_1 a_\ee^{-1} \int_{G_\ee^j} |v_\ee^j|^2 \le K |u|^2 \ee^{n}.
    \end{equation}
    Thus
    \begin{align}
    \frac{1}{|G_\ee^j|} \int_{G_\ee^j} |v_\ee^j(u,x')|^2 dx' &\le Ka_\ee^{1-n} \int_{G_\ee^j} |v_\ee^j(u,x')|^2 dx' \\
    &\le  Ka_\ee^{2-n} a_\ee^{-1} \int_{G_\ee^j} |v_\ee^j(u,x')|^2 dx' \\
    & \le K |u|^2 a_\ee^{2-n} \ee^n = K|u|^2 \ee.
    \end{align}
    This completes the proof.
\end{proof}

\subsection{Convergence to the \textquotedblleft strange
term\textquotedblright}

The following result plays a crucial rol in the proof of Theorem 1.

\begin{lemma}
\label{lemma ball conv} Let $H$ be the function defined by
\eqref{H
definition}, and let $\varphi $ an arbitrary function in $C^{\infty }(\Omega
)$. Then for any test function $h\in H^{1}(\Omega ,\Gamma _{2})$ we have
\begin{equation}
\bigg|\sum\limits_{j\in \Upsilon _{\varepsilon }}\int\limits_{{\ (\partial
T_{\varepsilon /4}^{j})^{+}}}\partial _{\nu _{x}}{\hat{w}_{\varepsilon }^{j}}%
(\varphi (\tilde{P}_{\varepsilon }^{j}),{s})h({s})ds+C_{0}^{n-2}\int\limits_{%
{\Gamma _{1}}}H(\varphi (x^{\prime }))h(x^{\prime })dx^{\prime }\bigg|%
\rightarrow 0,  \label{ball conv}
\end{equation}%
as $\varepsilon \rightarrow 0$, where $\tilde{P}_{\varepsilon }^{j}\in {%
G_{\varepsilon }^{j}}$ and {$\nu =(-1,0,\cdots ,0)$} is the unit outward
normal to ${(\partial T_{\varepsilon /4}^{j})^{+}}$.
\end{lemma}

\begin{proof}
Consider the cylinder
\begin{equation*}
Q^j_\ee = \left\{x\in\mathbb{R}^n \quad : \quad 0 < x_1 < \ee, \quad
-\frac \ee 2 < x_i - (\tilde P^j_{\ee})_i <\frac \ee 2, \quad
i=2,\cdots, {n}\right\}
\end{equation*}
We define the auxiliary function $%
\theta^j_\ee$ as the unique
solution to the following boundary value
problem
\begin{equation}\label{theta aux prob}
\begin{dcases}
\Delta
\theta^j_\ee = 0, & Y_\ee^j = Q^j_\ee\setminus\overline{{(T^{j}_{\ee/4})^+}}%
,\\
\pa_{{\nu}}\theta^j_\ee = {-}\pa_\nu { \hat w _\ee^j }(\varphi(\tilde
P_\ee^j), {x}), & x\in {(\pa T^{j}_{\ee/4})^+} ,\\
\frac{\pa\theta^j_\ee}{%
\pa x_1} = \mu^j_\ee, & x\in \gamma^j_\ee = \pa Q^j_\ee \cap \{x:\, x_1 =
\ee\},\\
\frac{\pa\theta^j_\ee}{\pa x_1} = 0, & \mbox{on the rest of the
boundary } \pa Q^j_\ee,\\
\langle\theta^j_\ee\rangle_{Y^j_\ee} = 0.
\end{%
dcases}
\end{equation}
The constant $\mu^j_\ee$ is defined from the
solvability condition for the problem \eqref{theta aux prob}
\begin{equation%
}
\mu^j_\ee = -C_0^{n-2} H(\varphi(\tilde
P^j_\ee)).
\end{equation}
We take $\theta^j_\ee$ as a test function in the
integral identity
associated to the problem \eqref{theta aux prob} and
obtain
\begin{equation}\label{theta int ident}
\int\limits_{Y^j_\ee}|%
\nabla\theta^j_\ee|^2 dx
=
-\mu^j_\ee\int\limits_{\gamma^{j}_{\ee}}\theta^j_\ee dx'
+
\int\limits_{S_{\ee/4}^{j, +}}\pa_\nu \hat w \theta^j_\ee ds.
\end{%
equation}
Using the embedding theorems we obtain the following estimate
\begin{equation}\label{gamma trace estim}
\int\limits_{\gamma^{j}_{\ee}}|%
\theta^j_\ee|dx'
\le
K\ee^{(n-1)/2}\Vert\theta^j_\ee\Vert_{L_2(\gamma^{j}_{\ee})}
\le
K\ee^{n/2}\Vert\nabla\theta^j_\ee\Vert_{L_2(Y^j_\ee)}.
\end{equation}%

Taking into account that
\begin{equation*} \max\limits_{{(\pa
T^{j}_{\ee/4})^+} }|\pa_\nu{\hat w_\ee^j(\varphi(\tilde P^j_\ee),  {x})}|\le
K\frac{a_\ee^{-1}}{\left|\frac{x - P^j_\ee}{a_\ee}\right|^{n-1}} = K
a^{n-2}_{\ee}\ee^{1-n}\le K
\end{equation*}
and using some estimates
proved in \cite{OlSh1996} we derive
\begin{align}
\int\limits_{{(\pa
T^{j}_{\ee/4})^+} }|(\pa_\nu {\hat w_\ee^j (\varphi(\tilde P^j_\ee), s) })
\theta^j_\ee| ds &\le K\int\limits_{{(\pa T^{j}_{\ee/4})^+}
}|\theta^j_\ee|ds \le K\ee^{(n-1)/2}\Vert\theta^j_\ee\Vert_{L_2({(\pa
T^{j}_{\ee/4})^+} )} \nonumber \\
&\le
K\ee^{(n-1)/2}\{\ee^{-1/2}\Vert\theta^j_\ee\Vert_{L_2(Y^j_\ee)} + \sqrt{\ee}%
\Vert\nabla\theta^j_\ee\Vert_{L_2(Y^j_\ee)}\} \nonumber\\
&\le
K\ee^{n/2}\Vert\nabla\theta^j_\ee\Vert_{L_2(Y^j_\ee)} \label{ball boundary
estim}
\end{align}
From the above estimates  \eqref{gamma trace estim} and
\eqref{ball
boundary estim} we get
\begin{equation}\label{theta grad estim}
\Vert\nabla\theta^j_\ee\Vert^2_{L_2(Y^j_\ee)} \le K\ee^{n}
\end{equation}%

From the estimate \eqref{theta grad estim} it follows that
\begin{%
equation*}
\sum\limits_{j\in\Upsilon_\ee}\Vert\theta^j_\ee\Vert^2_{L_2(Y^j_%
\ee)} \le K\ee.
\end{equation*}
Adding all the above integral identities
for problems \eqref{theta
aux prob} we derive that for $h\in H^1(\Omega)$
the following
inequality holds
\begin{align}
&\bigg|\sum\limits_{j\in%
\Upsilon_\ee}\int\limits_{{(\pa T^{j}_{\ee/4})^+} }\pa_{\nu}{ \hat w_\ee^j
(\varphi(\tilde P^j_\ee), s)} h ds +
C_0^{n-2}\int\limits_{\Gamma_1}H(\varphi(x')) {h}dx' \bigg| \nonumber \\
&
\quad \le
\bigg|\sum\limits_{j\in\Upsilon_\ee}\int\limits_{\hat
Y^j_\ee}\nabla%
\theta^j_\ee\nabla h dx
\bigg|+
\bigg|\sum\limits_{j\in\Upsilon_\ee}\mu^j_\ee\int\limits_{%
\gamma^{j}_{\ee}}h
dx' + C_0^{n-2}\int\limits_{\Gamma_1}H(\varphi(x'))h
dx'\bigg|.
\label{lemma theta estim 1}
\end{align}
Let us estimate the
terms in the right-hand side of the inequality
\eqref{lemma theta estim 1}.
By using the estimate \eqref{theta grad
estim}, we get the following
inequality on the first term
\begin{equation}\label{lemma 2 estim 1}%
	\bigg|\sum\limits_{j\in\Upsilon_\ee}\int\limits_{Y^j_\ee}\nabla\theta^j_%
	\ee\nabla
	h dx \bigg|\le K\sqrt{\ee}\Vert h\Vert_{H^1(\Omega).}
\end{equation}
Denote
\begin{equation}
\hat\gamma^j_\ee = {(Q^j_\ee)^0},
\qquad \Gamma^\ee_1 = \bigcup\limits_{j\in\Upsilon_\ee}\hat\gamma^j_\ee.
\end{equation}Then we have
\begin{align*}
&\left|\sum\limits_{j\in%
\Upsilon_\ee}\mu^j_\ee\int\limits_{\gamma^{j}_{\ee}}h dx' +
C_0^{n-2}\int\limits_{\Gamma_1}H(\varphi(x'))h
dx'\right|\\
&\quad\le\bigg|C_0^{n-2}\sum\limits_{j\in\Upsilon_\ee}\left(%
\int\limits_{\gamma^{-}_{j,\ee}}H(\varphi(\tilde P^j_\ee))h dx' -
\int\limits_{\hat\gamma^j_\ee}H(\varphi(x'))h dx' \right)\bigg|\\
&\quad\le
C_0^{n-2}\bigg|\sum\limits_{j\in\Upsilon_\ee}\int\limits_{\hat\gamma^j_%
\ee}(H(\varphi(\tilde P^j_\ee)) - H(\varphi(x')))h dx' \bigg|
\\
&\quad\quad+
C_0^{n-2}\bigg|\sum\limits_{j\in\Upsilon_\ee}\left(\int%
\limits_{\gamma^{j}_{\ee}}H(\varphi(\tilde
P^j_\ee))h dx' -
\int\limits_{\hat\gamma^j_\ee}H(\varphi(\tilde
P^j_\ee))h dx'
\right)\bigg|.
\end{align*}

Let us estimate the terms in the right-hand
side of the obtained
inequality. For the first term we have{
\begin{align*}%
\bigg|\sum\limits_{j\in\Upsilon_\ee}\int\limits_{\hat\gamma^j_\ee}(H(%
\varphi(\tilde P^j_\ee)) - H(\varphi(x')))h dx' \bigg|& \le K\Vert
h\Vert_{L_2(\Gamma^\ee_1)} \max_{\substack{j \in \Upsilon_\ee \\ x' \in
\hat\gamma^j_\ee}} \bigg|H(\varphi(\tilde P^j_\ee)) - H(\varphi(x'))
\bigg|\\
& \le K\Vert h\Vert_{L_2(\Gamma^\ee_1)} \max_{\substack{j \in
\Upsilon_\ee \\ x' \in \hat\gamma^j_\ee}} \bigg|\varphi(\tilde P^j_\ee) -
\varphi(x') \bigg|\\
& \le K \ee\Vert h\Vert_{H^1(\Omega)}.
\end{align*}%

} By using continuity in $L^2$-norm on the hyperplanes of the
functions
from $H^1(\Omega)$ we estimate the second term
\begin{gather*}%
\bigg|\sum\limits_{j\in\Upsilon_\ee}\left(\int\limits_{\gamma^{j}_{\ee}}H(%
\varphi(\tilde
P^j_\ee))h dx' -
\int\limits_{\hat\gamma^j_\ee}H(\varphi(\tilde
P^j_\ee))h dx' \right)\bigg|
\le K\sqrt{\ee}\Vert
h\Vert_{H^1(\Omega)}.
\end{gather*}
Hence, we have
\begin{equation}\label{lemma 2 estim 2}
\bigg|\sum\limits_{j\in\Upsilon_%
\ee}\mu^j_\ee\int\limits_{\gamma^{j}_{\ee}}h
dx' +
C_0^{n-2}\int\limits_{\Gamma_1}H(\varphi(x'))h dx'\bigg|\le
K\sqrt{\ee}%
\Vert h\Vert_{H^1(\Omega)}.
\end{equation}
Combining estimates \eqref{%
lemma 2 estim 1} and \eqref{lemma 2 estim
2} we conclude the proof.
\end{proof}

\section{Convergence of the boundary integrals where $u_\protect\varepsilon %
\le 0$}

\label{sec:convergence boundary integrals uee less 0}

\subsection{The auxiliary function $\protect\kappa_\protect\varepsilon^j$}

We introduce the function $\kappa _{\varepsilon }^{j}$ as the unique
solution of the following problem
\begin{equation}
\left\{
\begin{array}{ll}
\Delta \kappa _{\varepsilon }^{j}=0, & x\in {(T_{\varepsilon /4}^{j})^{+}}%
\setminus \overline{G_{\varepsilon }^{j}}, \\
\kappa _{\varepsilon }^{j}=1, & x\in G_{\varepsilon }^{j}, \\
{\partial _{\nu }\kappa _{\varepsilon }^{j}=0} & {x\in (\partial
T_{\varepsilon /4}^{j})^{0}\setminus \overline{G_{\varepsilon }^{j}}}, \\
\kappa _{\varepsilon }^{j}=0, & x\in \partial T_{\varepsilon /4}^{j},%
\end{array}%
\right.   \label{kappa prob}
\end{equation}%
and then we define
\begin{equation}
\kappa _{\varepsilon }=\left\{
\begin{array}{ll}
\kappa _{\varepsilon }^{j}(x), & x\in {(T_{\varepsilon /4}^{j})^{+}},\,j\in
\Upsilon _{\varepsilon } \\
0, & x\in \mathbb{R}^{n}\setminus \bigcup\limits_{j\in \Upsilon
_{\varepsilon }}{\overline{(T_{\varepsilon /4}^{j})^{+}}}.%
\end{array}%
\right.   \label{Kappa def}
\end{equation}%
It is easy to see that $\kappa _{\varepsilon }\in H^{1}(\Omega )$ and
\begin{equation}
\kappa _{\varepsilon }\rightharpoonup 0\mbox{ weakly in }H^{1}(\Omega )%
\mbox{
as }\varepsilon \rightarrow 0.  \label{kappa conv}
\end{equation}

\subsection{Estimate of the difference between $\protect\kappa _{\protect%
\varepsilon }^{j}$ and $\hat{\protect\kappa}_{\protect\varepsilon }^{j}$}

\begin{lemma}
\label{lemma estim kappas} Let $\kappa _{\varepsilon }^{j}$ and $\hat{\kappa}
$ as above. Then
\begin{equation}
\sum\limits_{j\in \Upsilon _{\varepsilon }}\Vert \kappa _{\varepsilon }^{j}-{%
\ \hat{\kappa}_{\varepsilon }^{j}}\Vert _{H^{1}({(T_{\varepsilon /4}^{j})^{+}%
})}^{2}\leq K\varepsilon .  \label{estim kappas}
\end{equation}
\end{lemma}

\begin{proof}
The function $v^j_\ee = \kappa^j_\ee - \hat \kappa$ satisfies the
following problem
\begin{equation*}
\begin{dcases}
\Delta v^j_\ee = 0 & x\in {(T^j_{\ee/4})^+},\\
{v_\ee^j = 0 } & { x \in G^j _\ee,} \\
\pa_{\nu}v^j_\ee = 0 & x\in { (T _\ee^j )^0 \setminus \overline{G^j_\ee}},\\
v^j_\ee = -{ \hat \kappa_\ee^j } & x\in {(\pa T^j_{\ee/4})^+}.
\end{dcases}
\end{equation*}
We take $v^j_\ee$ as a test function in an integral identity for the
above problem
\begin{equation*}
\Vert\nabla v^j_\ee \Vert_{L_2({(T^j_{\ee/4})^+})} = -
\int\limits_{\pa {(T^j_{\ee/4})^+}} \pa_{\nu}v^j_\ee \hat \kappa ds.
\end{equation*}
We transform the right-hand side expression of the identity in the
following way:
\begin{equation*}
- \int\limits_{\pa {(T^j_{\ee/4})^+}} \pa_{\nu}v^j_\ee \hat \kappa ds = -\int\limits_{{(T^j_{\ee/4})^+}\setminus\overline{T^{j}_{\ee/8}}}\nabla v^j_\ee\nabla\hat \kappa dx + \int\limits_{\pa T^{j}_{\ee/8}}\pa_\nu v^j_\ee \hat \kappa ds.
\end{equation*}
For an arbitrary point $x_0 \in \pa T^{j}_{\ee/8}$ we have
\begin{align*}
|\pa_{x_i}v^j_\ee(x_0)| &\le \frac{1}{|T_{\ee/16}(x_0)|} \,\bigg|\int\limits_{T_{\ee/16}(x_0)}\frac{\pa v^j_\ee}{\pa x_i} dx\bigg| \\
&= \frac{K}{\ee^n}\bigg|\int\limits_{\pa T_{\ee/16}(x_0)}v^j_\ee\nu_i dx\bigg| \le K.
\end{align*}
Last estimate implies that $|\nabla v^j_\ee(\tilde x)|\le K$ for
$\tilde x \in \pa T^j_{\ee/8}$. Therefore, we can estimate the
second term in the following way
\begin{equation*}
\bigg|\int\limits_{\pa T^{j}_{\ee/8}}\pa_\nu v^j_\ee \hat \kappa
ds\bigg| \le K\max\limits_{\pa T^{j}_{\ee/8}}|\hat \kappa||\pa
T^{j}_{\ee/8}|\le K\ee^{n}.
\end{equation*}
Then we estimate the first term
\begin{equation*}
\bigg|\int\limits_{{(T^j_{\ee/4})^+}\setminus\overline{T^{j}_{\ee/8}}}\nabla v^j_\ee\nabla\hat \kappa dx \bigg|\le K\ee a^{n-2}_\ee = K \ee^n
\end{equation*}
By combining acquired estimations we derive
\begin{equation*}
\Vert\nabla v^j_\ee\Vert^2_{L_2({(T^j_{\ee/4})^+})}\le K\ee^{n}
\end{equation*}
Friedrichs' inequality imply
\begin{equation*}
\Vert v^j_\ee\Vert^2_{L_2({(T^j_{\ee/4})^+})}\le K\ee^{n+2}
\end{equation*}
This concludes the proof.
\end{proof}

\subsection{Convergence to the strange term}

\begin{lemma}
\label{lemma aux hom kappa} {Let $\lambda _{G_{0}}$ be given by %
\eqref{eq:lambda defn}.} Then for all functions $h\in H^{1}(\Omega ,\Gamma
_{2})$ we have
\begin{equation}
\bigg|\sum\limits_{j\in \Upsilon _{\varepsilon }}\int\limits_{{\ (\partial
T_{\varepsilon /4}^{j})^{+}}\cap \Omega }\partial _{\nu _{x}}{\hat{\kappa}%
_{\varepsilon }^{j}(s)}h({s})ds+C_{0}^{n-2}\lambda _{G_{0}}\int\limits_{{%
\Gamma _{1}}}hdx^{\prime }\bigg|\rightarrow 0,
\end{equation}%
as $\varepsilon \rightarrow 0$, where $\nu $ is the unit outward normal to $%
\partial T_{\varepsilon /4}^{j}\cap \Omega $.
\end{lemma}

\begin{proof}
By analogy with the proof of Lemma \ref{lemma ball conv}, we define the function $\theta^j_\ee$ as a solution to the following boundary value problem
\begin{equation}\label{theta aux prob 2}
\begin{dcases}
    \Delta \theta^j_\ee = 0, & Y^j_\ee,\\
    \pa_{{\nu}}\theta^j_\ee = {-}\pa_\nu { \hat \kappa_\ee^j }, & x\in {(\pa T^{j}_{\ee/4})^+} ,\\
    \frac{\pa\theta^j_\ee}{\pa x_1} = \mu, & x\in \gamma^j_\ee,\\
    \frac{\pa\theta^j_\ee}{\pa x_1} = 0, & \pa Q^j_\ee \setminus ( {(\pa T^{j}_{\ee/4})^+}  \cup   \gamma^j_\ee  ),\\
    \langle\theta^j_\ee\rangle_{Y^j_\ee} = 0,
\end{dcases}
\end{equation}
The constant $\mu$ is defined from the solvability condition for the problem \eqref{theta aux prob 2}
\begin{equation}\label{mu epsilon def}
\mu = -C_0^{n-2}\lambda.
\end{equation}
By using the same technique as in the proof of the Lemma~\ref{ball conv} we have
\begin{equation}
\Vert\nabla\theta^j_\ee\Vert^2_{L_2(Y^j_\ee)} \le K\ee^{n},\qquad \sum\limits_{j\in\Upsilon_\ee}\Vert\theta^j_\ee\Vert^2_{L_2(Y^j_\ee)} \le K\ee.
\end{equation}
Summing up all integral identities for the problems \eqref{theta aux prob 2} we derive that for the arbitrary function from $H^1(\Omega)$ the following inequality is true
\begin{align}
&\bigg|\sum\limits_{j\in\Upsilon_\ee}\int\limits_{{(\pa T^{j}_{\ee/4})^+} }(\pa_{\nu} {\hat \kappa_\ee^j}) h ds + C_0^{n-2}\lambda\int\limits_{\Gamma_1}hdx' \bigg| \nonumber\\
&\quad \le \bigg|\sum\limits_{j\in\Upsilon_\ee}\int\limits_{\hat Y^j_\ee}\nabla\theta^j_\ee\nabla h dx \bigg|+ \bigg|\sum\limits_{j\in\Upsilon_\ee}\mu\int\limits_{\gamma^{j}_{\ee}}h dx' + C_0^{n-2}\lambda\int\limits_{\Gamma_1}h dx'\bigg|, \label{lemma theta estim 1 2}
\end{align}
Let us estimate terms in the right-hand side of the inequality \eqref{lemma theta estim 1 2}.
By using the estimate \eqref{theta grad estim}, we get following estimation of the first term
\begin{equation}\label{lemma 2 estim 1 2}
\bigg|\sum\limits_{j\in\Upsilon_\ee}\int\limits_{Y^j_\ee}\nabla\theta^j_\ee\nabla h dx \bigg|\le K\sqrt{\ee}\Vert h\Vert_{H^1(\Omega)}.
\end{equation}
Then we have
\begin{gather*}
\bigg|\sum\limits_{j\in\Upsilon_\ee}\mu\int\limits_{\gamma^{j}_{\ee}}h dx' + C_0^{n-2}\lambda\int\limits_{\Gamma_1}h dx'\bigg|\le\bigg|C_0^{n-2}\lambda\sum\limits_{j\in\Upsilon_\ee}\left(\int\limits_{\gamma^{-}_{j,\ee}}h dx' - \int\limits_{\hat\gamma^j_\ee}h dx' \right)\bigg|
\end{gather*}
By using continuity in $L_2$-norm on the hyperplanes of the functions from $H^1(\Omega)$ we estimate the second term
\begin{gather*}
\bigg|\sum\limits_{j\in\Upsilon_\ee}\left(\int\limits_{\gamma^{j}_{\ee}}h dx' - \int\limits_{\hat\gamma^j_\ee} h dx' \right)\bigg| \le K\sqrt{\ee}\Vert h\Vert_{H^1(\Omega)}
\end{gather*}
Hence, we have
\begin{equation}\label{lemma 2 estim 2 2}
\bigg|\sum\limits_{j\in\Upsilon_\ee}\mu_\ee\int\limits_{\gamma^{j}_{\ee}}h dx' + C_0^{n-2}\int\limits_{\Gamma_1} h dx'\bigg|\le K\sqrt{\ee}\Vert h\Vert_{H^1(\Omega)}
\end{equation}
Combining estimations \eqref{lemma 2 estim 1 2} and \eqref{lemma 2 estim 2 2} we conclude the proof.
\end{proof}

\section{Proof of Theorem \protect\ref{thm:1}}

For different reasons it is convenient to introduce some new notation:
instead to use the decomposition $u_{0}=u_{0,+}-u_{0,-}^{{}}$ mentioned at
the introduction \ (see the statement of Theorem 1) we shall use the
alternative decomposition $u_{0}=u_{0}^{+}+u_{0}^{-}$ (i.e. $%
u_{0}^{+}=u_{0,+}$ but $u_{0}^{-}=-u_{0,-}^{{}}$).

\begin{proof}
Let $\varphi(x)$ be an arbitrary function from $C^{\infty}_0(\Omega)$. We choose point $\hat P^j_\ee \in \overline{G^j_\ee}$ such that \begin{equation*}
\min\limits_{x\in \overline{G^j_\ee}}\varphi^{+}(x) = \varphi^{+}(\hat P^j_\ee),
\end{equation*}
where $\varphi^{+} = \max\{0, \varphi(x)\}$ and $\varphi^{-}(x) = \varphi(x) - \varphi^{+}(x)$. Define the function
\begin{equation}\label{hom theor aux func W}
\mathcal W_\ee(\varphi^{+}, x) =
\left\{
\begin{array}{ll}
w^j_\ee(\varphi^{+}(\hat P^j_\ee), x), & x\in (T^{j}_{\ee/4})^{+},\quad j\in\Upsilon_\ee\\
0, & x\in \mathbb{R}^n\setminus\bigcup\limits_{j\in\Upsilon_\ee}\overline{(T^{j}_{\ee/4})^{+}}.
\end{array}
\right.
\end{equation}
From estimates \eqref{W estim} we conclude that
\begin{equation}\label{mathcal W conv}
\mathcal W_\ee(\varphi^{+}, x) \rightharpoonup 0
\end{equation}
in $H^1(\Omega, \Gamma_2)$ as $\ee\rightarrow 0$. We set
\begin{equation*}
    v = \varphi^{+} - \mathcal W_\ee(\varphi^{+}, x) + (1 - \kappa_\ee)\varphi^{-}
\end{equation*}
as a test function in integral inequality \eqref{integral inequality} where $\varphi$ is an arbitrary function from $C^{\infty}_0(\Omega)$. Notice that $v \in K_\ee$. Indeed, according to the Lemma~\ref{aux fun w estim parameter} and using that $\kappa_\ee \equiv 1$ in $G_\ee$ we have for all $x\in G^j_\ee$ that
\begin{equation}
v = \varphi^{+} - {\mathcal W_\ee(\varphi^{+}, x)} + (1 - \kappa_\ee)\varphi^{-} \ge \varphi^{+}(\hat P^j_\ee) - w^j_\ee(\varphi^{+}(\hat P^j_\ee), x) \ge 0.
\end{equation}
Hence, we get
\begin{align}
&\int\limits_{\Omega}\Bigg\{ \nabla\Big(\varphi^{+} - \mathcal W_\ee({\varphi^{+}}, x) + (1 - \kappa_\ee)\varphi^{-}\Big)\\
&\qquad \qquad \cdot \nabla\Big(\varphi^{+} - \mathcal W_\ee({\varphi^{+}}, x) + (1 - \kappa_\ee)\varphi^{-} - u_\ee\Big) \Bigg\} dx +\\
&\qquad +\ee^{-k}\sum\limits_{j\in\Upsilon_\ee}\int\limits_{G^j_\ee}\sigma(\varphi^{+} - w^j_\ee(\varphi^{+}(\hat P^j_\ee), x))(\varphi^{+} - w^j_\ee(\varphi^{+}(\hat P^j_\ee), x) - u_\ee)dx'\ge\\
&\quad \ge\int\limits_{\Omega}f(\varphi^{+} - \mathcal W_\ee(\varphi^{+}, x) + (1 - \kappa_\ee)\varphi^{-} - u_\ee)dx.
\end{align}
Considering the first integral of the right-hand side of the inequality above we have:
\begin{align}
&\int\limits_{\Omega}\nabla(\varphi - \mathcal W_\ee({\varphi^{+}}, x) - \kappa_\ee\varphi^{-})\nabla(\varphi - \mathcal W_\ee({\varphi^{+}}, x) - \kappa_\ee\varphi^{-} - u_\ee)dx \nonumber\\
&\quad=\int\limits_{\Omega}\nabla\varphi\nabla(\varphi - \mathcal W_\ee(\varphi^{+}(\hat P^{j}_\ee), x) - \kappa_\ee\varphi^{-} - u_\ee)dx - \nonumber\\
& \qquad -\int\limits_{\Omega}\nabla\mathcal W_\ee({\varphi^{+}}, x) \nabla(\varphi - \mathcal W_\ee({\varphi^{+}}, x) - \kappa_\ee\varphi^{-} - u_\ee) dx  \nonumber\\
& \qquad -\int\limits_{\Omega}\nabla(\kappa_\ee\varphi^{-}) \nabla(\varphi - \mathcal W_\ee({\varphi^{+}}, x) - \kappa_\ee\varphi^{-} - u_\ee) dx \nonumber \\
	\label{J1 J2 J3 introduction}
& \quad = \sum\limits_{i = 1}^{3}J^i_\ee. 
\end{align}

By using \eqref{kappa conv} and \eqref{mathcal W conv} we have
\begin{equation}\label{J_1 limit}
\lim\limits_{\ee\rightarrow0} J^1_{\ee} = \int\limits_{\Omega}\nabla\varphi\nabla(\varphi - u_0) dx.
\end{equation}

Then we proceed by transforming $J^2_\ee$ in the following way
\begin{align}
J^2_\ee &= -\sum\limits_{j\in\Upsilon_\ee}\int\limits_{{(T^j_{\ee/4})^+}}\nabla w^j_{\ee}(\varphi^{+}(\hat P^j_\ee), x) \cdot \nabla(\varphi - w^j_\ee(\varphi^{+}(\hat P^j_\ee), x) - \kappa_\ee\varphi^{-} - u_\ee)dx \nonumber \\
& = - \sum\limits_{j\in\Upsilon_\ee}\int\limits_{{(T^j_{\ee/4})^+}} \Bigg \{ \nabla \Big (w^j_{\ee}(\varphi^{+}(\hat P^j_\ee), x) - { \hat w _\ee^j (\varphi^{+}(\hat P^j_\ee), x)} \Big) \nonumber\\
&\qquad \qquad \qquad \qquad \cdot \nabla\Big (\varphi - w^j_\ee(\varphi^{+}(\hat P^j_\ee), x) - \kappa_\ee\varphi^{-} - u_\ee \Big) \Bigg \} dx -\nonumber \\
&\quad -\sum\limits_{j\in\Upsilon_\ee}\int\limits_{{(T^j_{\ee/4})^+}}\nabla {\hat w_\ee^j(\varphi^{+}(\hat P^j_\ee), x)} \nabla(\varphi - w^j_\ee(\varphi^{+}(\hat P^j_\ee), x) - \kappa_\ee\varphi^{-} - u_\ee) dx \nonumber\\
&= I^1_\ee + I^2_\ee.
\label{J2 decomposition}
\end{align}
Lemma~\ref{lemma w and hat w} implies that
\begin{equation}\label{I1 limit}
I^1_\ee \rightarrow 0 \mbox{ as } \ee\rightarrow 0.
\end{equation}
By using Green's formula we have the following decomposition of the second integral
\begin{align}
I^2_\ee &= - \sum\limits_{j\in\Upsilon_\ee}\int\limits_{{(\partial T^j_{\ee/4})^+}} \Bigg \{\pa_{\nu}\hat w^j_\ee(\varphi^{+}(\hat P^j_\ee), x)  (\varphi^{+} - w^j_\ee(\varphi^{+}(\hat P^j_\ee), x) - u_\ee) \Bigg\} ds \nonumber \\
&\quad - \ee^{-k}\sum\limits_{j\in\Upsilon_\ee}\int\limits_{G^j_\ee} \Bigg\{ \sigma\Big(\varphi^{+}(\hat P
^j_\ee) - { \hat w _\ee^j (\varphi^{+}(\hat P^j_\ee), x)}\Big)  \Big(\varphi^{+} - w^j_\ee(\varphi^{+}(\hat P^j_\ee), x) - u_\ee \Big) \Bigg\} dx' \nonumber \\
&=\mathcal I^1_\ee + \mathcal I^2_\ee.
\label{I2 decomposition}
\end{align}
From Lemma~\ref{lemma ball conv} we have
\begin{equation}\label{mathcal I1 limit}
\lim\limits_{\ee\rightarrow 0}\mathcal I^1_\ee = C_0^{n-2}\int\limits_{\Omega}H(\varphi^{+}(x))(\varphi^{+} - u_0)dx.
\end{equation}
Combining \eqref{J2 decomposition}-\eqref{mathcal I1 limit} implies that
\begin{equation}\label{J2 limit}
\lim\limits_{\ee\rightarrow0}J^2_\ee = C_0^{n-2}\int\limits_{\Omega}H(\varphi^{+}(x))(\varphi^{+} - u_0)dx +  {\lim_{\ee \to 0} \mathcal I^2_\ee }.
\end{equation}
Now we consider third term of the identity~\eqref{J1 J2 J3 introduction}. { By using the fact that
\begin{align*}
    \nabla (\kappa_\ee^j \varphi^-) \cdot \nabla\rho_\ee &= \left( \nabla \kappa_\ee^j \cdot \nabla \rho_\ee \right) \varphi^-+ \kappa_\ee^j \nabla \varphi^- \cdot \nabla \rho_\ee \\
    &=\nabla \kappa_\ee^j \cdot \nabla (\varphi^- \rho_\ee) - (\nabla \kappa_\ee^j \cdot \nabla \varphi^-)\rho_\ee +  \kappa_\ee^j \nabla \varphi^- \cdot \nabla \rho_\ee\\
    &=\nabla \hat \kappa_\ee^j \cdot \nabla (\varphi^- \rho_\ee)-\nabla (\hat \kappa_\ee^j - \kappa_\ee^j) \cdot \nabla (\varphi^- \rho_\ee)  \\
    &\quad - (\nabla \kappa_\ee^j \cdot \nabla \varphi^-)\rho_\ee +  \kappa_\ee^j \nabla \varphi^- \cdot \nabla \rho_\ee.
\end{align*}
we deduce that}
\begin{align}
J^3_\ee &= - \sum\limits_{j\in\Upsilon_\ee}\int\limits_{{ (T^j_{\ee/4})^+}}\nabla { \hat\kappa_\ee^j} \nabla\Big(\varphi^{-}(\varphi - \mathcal W_\ee(\varphi^{+}, x) - \kappa_\ee\varphi^{-} - u_\ee)\Big)dx \nonumber \\
& \quad -\sum\limits_{j\in\Upsilon_\ee}\int\limits_{{ (T^j_{\ee/4})^+}}\nabla(\kappa^j_\ee - { \hat\kappa_\ee^j}) \nabla\Big(\varphi^{-}(\varphi - \mathcal
W_\ee(\varphi^{+}, x) - \kappa_\ee\varphi^{-} - u_\ee)\Big)dx \nonumber \\
&\quad + \int\limits_{\Omega}\Big( \nabla\kappa_\ee \cdot \nabla\varphi^{-} \Big) \Big(\varphi^{-}(\varphi - \mathcal W_\ee(\varphi^{+}, x) - \kappa_\ee\varphi^{-} - u_\ee) \Big)dx \nonumber\\
& \quad - \int\limits_{\Omega}{ \kappa_\ee}\nabla\varphi^{-} \cdot \nabla\Big(\varphi^{-}(\varphi - \mathcal W_\ee(\varphi^{+}, x) - \kappa_\ee\varphi^{-} - u_\ee) \Big) dx \nonumber \\
&= \mathcal Q^1_\ee + \mathcal Q^2_\ee + \mathcal Q^3_\ee + \mathcal Q^4_\ee.
\label{J3 decomposition}
\end{align}
Lemma~\ref{lemma estim kappas} implies that $\mathcal Q^2_\ee \rightarrow 0$ as $\ee\rightarrow 0$. Then we use \eqref{u epsilon estim}, \eqref{u0 def}, \eqref{kappa conv} and \eqref{mathcal W conv} to derive that $\mathcal Q^3_\ee \rightarrow 0$ and $\mathcal Q^4_\ee\rightarrow0$ as $\ee\rightarrow0$. Then we transform $\mathcal Q^1_\ee$ using the Green's formula
\begin{equation}\label{Q1 decomposition}
\begin{gathered}
\mathcal Q^1_\ee = -\sum\limits_{j\in\Upsilon_\ee}\int\limits_{{ (\pa T^j_{\ee/4})^+}}(\pa_\nu{ \hat\kappa_\ee^j}) \varphi^{-}(\varphi - \mathcal W_\ee(\varphi^{+}, x) - \kappa_\ee\varphi^{-} - u_\ee) ds -\\
-\sum\limits_{j\in\Upsilon_\ee}\int\limits_{{G_\ee^j}}(\pa_\nu{ \hat\kappa_\ee^j}) \varphi^{-}(\varphi - \mathcal W_\ee(\varphi^{+}, x) - \kappa_\ee\varphi^{-} - u_\ee) dx'.
\end{gathered}
\end{equation}
By using the fact that $\pa_\nu{\hat\kappa \ge 0}$ and $\varphi^{-}\mathcal W_\ee(\varphi^{+}, x) \le 0, \, \varphi^{-}u_{\ee} \le 0$ a.e. in $G_\ee$ we have that
\begin{equation}
\begin{gathered}
-\sum\limits_{j\in\Upsilon_\ee}\int\limits_{{ G_\ee^j}}(\pa_\nu{ \hat\kappa_\ee^j}) \varphi^{-}(\varphi - \mathcal W_\ee(\varphi^{+}, x) - \kappa_\ee\varphi^{-} - u_\ee) dx' \le 0.
\end{gathered}
\end{equation}
Hence, we conclude
\begin{equation}\label{Q1 estim}
\mathcal Q^1_\ee \le -\sum\limits_{j\in\Upsilon_\ee}\int\limits_{\pa T^j_{\ee/4}\cap\Omega}(\pa_\nu{ \hat\kappa_\ee^j}. )\varphi^{-}(\varphi - u_\ee) ds
\end{equation}
Lemma~\ref{lemma aux hom kappa} implies
\begin{equation}\label{J3 limit}
\lim\limits_{\ee\rightarrow 0}J^3_{\ee} = \lim\limits_{\ee\rightarrow 0}\mathcal Q^1_\ee \le \lambda C^{n-2}_0\int\limits_{\Omega}\varphi^{-}(\varphi - u_\ee) dx
\end{equation}
From \eqref{J1 J2 J3 introduction}, \eqref{J_1 limit}, \eqref{J2 limit}, \eqref{J3 limit} we derive that
\begin{align}
&\int\limits_{\Omega}\nabla\varphi \nabla(\varphi - u_0)dx + C^{n - 2}_0 \int\limits_{\Gamma_1}H(\varphi^{+})(\varphi - u_0)dx' + \lambda C^{n - 2}_0 \int\limits_{\Gamma_1}\varphi^{-}(\varphi - u_0)dx' \nonumber \\
&\qquad + \lim\limits_{\ee\rightarrow 0}\Bigg\{ \ee^{-k}\sum\limits_{j\in\Upsilon_\ee}\int\limits_{G^j_\ee}\sigma\left( \varphi^{+} - \mathcal W_\ee({\varphi^{+}}, x)\right) (\varphi^{+} - {\mathcal W_\ee({\varphi^{+}}, x)} - u_\ee)dx'  \nonumber \\
&\qquad\qquad  -\ee^{-k}\sum\limits_{j\in\Upsilon_\ee}\int\limits_{G^j_\ee}\sigma\left (\varphi^{+}(\hat P^j_\ee) - {\hat  w^j_\ee(\varphi^{+}(\hat P^j_\ee), x)} \right )(\varphi^{+} - {\mathcal W_\ee({\varphi^{+}}, x)} - u_\ee)  dx'\Bigg \} \nonumber \\
&\quad \ge \int\limits_{\Omega}f (\varphi - u_0)dx.
\label{pre-final inequality}
\end{align}
{We first notice that $\varphi^{+} - w^j_\ee(\varphi^{+}(\hat P^j_\ee),x) \ge \varphi^{+}(\hat P_\ee^j) - \widehat w^j_\ee(\varphi^{+}(\hat P^j_\ee,x))$ and so
\begin{equation*}
     \Bigg\{ \sigma\left( \varphi^{+} - w^j_\ee(\varphi^{+}(\hat P^j_\ee), x)\right) - \sigma\left (\varphi^{+}(\hat P^j_\ee) - {\hat  w^j_\ee(\varphi^{+}(\hat P^j_\ee), x)} \right ) \Bigg\} u_\ee \ge 0.
\end{equation*}
Thus, we only need to study the term
\begin{equation*}
    \ee^{-k} \int\limits_{G_\ee^j }\Bigg\{ \sigma\left( \varphi^{+} - w^j_\ee(\varphi^{+}(\hat P^j_\ee), x)\right) - \sigma\left (\varphi^{+}(\hat P^j_\ee) - {\hat  w^j_\ee(\varphi^{+}(\hat P^j_\ee), x)} \right ) \Bigg\} \Big(\varphi^{+} - w^j_\ee(\varphi^{+}(\hat P^j_\ee), x) \Big) dx'.
\end{equation*}
On the other hand,
\begin{align*}
    |\varphi^+ - {\mathcal W_\ee({\varphi^{+}}, x)}| &\le |\varphi^+| \le K, \\
    w^j_\ee(\varphi^{+}(\hat P^j_\ee), x) \le \hat w^j_\ee(\varphi^{+}(\hat P^j_\ee), x)| &\le |\varphi^+| \le K.
\end{align*}
Since $\sigma$ is Hölder continuous, and $\varphi \in \mathcal C^1 (\overline {\Omega})$},  we have that, for a.e. $x\in G^j_\ee$, 
\begin{align*}
\Big |\sigma(\varphi^{+} - {\mathcal W_\ee({\varphi^{+}}, x)}) &- \sigma(\varphi^{+}(\hat P^j_\ee) - {\mathcal W_\ee({\varphi^{+}}, x)}) \Big|\\
&\le K { \sum_{i=1}^2 |\varphi^{+}(x) - \varphi^{+}(\hat P^j_\ee){|^{\rho_i}} }\le K { \sum_{i=1}^2 |x - \hat P^j_\ee{|^{\rho_i}}}\\
&= K {\sum_{i=1}^2 {a_\ee^{\rho_i}}}.
\end{align*}
By using the same reasoning, estimate \eqref{W estim} implies that
\begin{gather}
\bigg|\ee^{-k}\sum\limits_{j\in\Upsilon_\ee}\int\limits_{G^j_\ee}\Big (\sigma(\varphi^{+} - {\mathcal W_\ee({\varphi^{+}}, x)}) - \sigma(\varphi^{+}(\hat P^j_\ee) - {\mathcal W_\ee({\varphi^{+}}, x)})\Big){\Big(\varphi^{+} - w^j_\ee(\varphi^{+}(\hat P^j_\ee), x) \Big)} dx'\bigg| \nonumber \\
\le K_2 { \ee^{-k} |G_\ee|} {\sum_{i=1}^2 a_\ee^{\rho_i}} \le K { \sum_{i=1}^2 a_\ee^{\rho_i}}.
\label{eq:use Holder 1}
\end{gather}
Then by Lemma~\ref{lemma w and hat w} we have that

\begin{align}
\ee^{-k}&\bigg|\sum\limits_{j\in\Upsilon_\ee}\int\limits_{G^j_\ee}\Big(\sigma(\varphi^{+}(\hat P^j_\ee) - {\mathcal W_\ee({\varphi^{+}}, x)}) - \sigma(\varphi^{+}(\hat P^j_\ee) - {\hat w_\ee^j})\Big)\Big(\varphi^{+} - {\mathcal W_\ee({\varphi^{+}}, x)} \Big) dx' \bigg| \nonumber\\
&\le K\ee^{-k}\sum\limits_{j\in\Upsilon_\ee}{ \sum_{i=1}^2 \int\limits_{G^j_\ee}|v(\varphi^{+}(\hat P^j_\ee), x) |^{\rho_i} dx' } \nonumber\\
&\le K\ee^{-k}\sum\limits_{j\in\Upsilon_\ee}|G_\ee^j| { \sum_{i=1}^2  \frac{1}{|G_\ee^j|} \int\limits_{G^j_\ee}|v(\varphi^{+}(\hat P^j_\ee), x) |^{\rho_i} dx'} \nonumber\\
\intertext{which, applying the $L^2(G_0) \to L^{\rho_i} (G_0)$ embedding for $0 < \rho_i \le 2$, can be estimated as}
...&\le K \ee^{-k}\sum\limits_{j\in{\Upsilon_\ee}} |G_\ee^j| { \sum_{i=1}^2 \left( \frac{1}{|G_\ee^j| } \int_{G_\ee^j} | v^j_\ee (\varphi^{+}(\hat P^j_\ee), x) |^2 dx ' \right) ^{\frac {\rho_i} 2}  } \nonumber \\
\intertext{and, by {Lemma \ref{lem:trace theorem vee}}}
...&\le K \ee^{-k}\sum\limits_{j\in{\Upsilon_\ee}} |G_\ee^j|{ \sum_{i=1}^2 \ee ^ {\frac {\rho_i} 2 } } \le K \ee^{-k}|\Upsilon_{\ee}||G_\ee^j| { \sum_{i=1}^2 \ee ^ {\frac {\rho_i} 2 } } \le K \ee^{-k + 1 - n + k(n-1)}{ \sum_{i=1}^2 \ee ^ {\frac {\rho_i} 2 } } \nonumber \\
&= K { \sum_{i=1}^2 \ee ^ {\frac {\rho_i} 2 } } \to 0,
\label{estim second term}
\end{align}
by using that $0 < \rho \le 2$.
Combining these estimates with \eqref{pre-final inequality} we derive{, since $\rho > 0$,} that
\begin{align}
\int\limits_{\Omega}\nabla\varphi\nabla(\varphi - u_0)dx &+ C^{n-2}_0\int\limits_{\Gamma_1}H(\varphi^{+})(\varphi - u_0)dx' +  \lambda C^{n-2}_0\int\limits_{\Gamma_1}\varphi^{-}(\varphi - u_0) \nonumber \\
&\ge \int\limits_{\Omega}f(\varphi - u_0)dx
\label{theorem int ident}
\end{align}
holds for any $\varphi \in H^1(\Omega, \Gamma_2)$.

Finally, given $\psi\in H^1(\Omega, \Gamma_2)$ we consider the test function $\varphi = u_0 \pm {\delta} \psi$, ${\delta}>0$ in \eqref{theorem int ident} and we pass to the limit as ${\delta}\rightarrow0$. By doing so we get that $u_0$ satisfies the integral condition
\begin{equation}
\int\limits_{\Omega}\nabla u_0 \nabla \psi dx + C^{n-2}_0 \int\limits_{\Gamma_1}H(u^{+}_0)\psi dx' + \lambda C^{n-2}_0\int\limits_{\Gamma_1}u^{-}_0\psi dx' = \int\limits_{\Omega}f\psi dx
\end{equation}
for any $\psi\in H^1(\Omega, \Gamma_2)$. This concludes the proof.
\end{proof}

\section{Possible extensions and comments}

\subsection{Extension to the case of $\protect\sigma $ as a maximal monotone graph}
\label{sec:mmg}

In \cite{DiGoPoSh Nonlin Anal} the authors showed that a similar problem,
although restricted to the case of spherical particles distributed through the whole domain, could be treated in
the general framework of maximal monotone graphs $\sigma$, which allow for a
\emph{common roof} between the Dirichlet, Neumann and Signorini boundary
conditions and many more.
We have restricted here to the
case of Hölder continuous $\sigma$ (see \eqref{eq:sigma regularity}) but
this condition is only used at very end, in estimates \eqref{eq:use Holder
1} and \eqref{estim second term} to compute the last term of
\eqref{pre-final inequality}. The superlinearity condition is only used to
obtain Lemma \ref{lem:trace theorem vee}. These seem to be only technical
difficulties, and can probably be avoided. Let us introduce what results
can be expected, if these problems could be circumvented.

\paragraph{Maximal monotone graph of $\mathbb R^2$} 
	A monotone graph of $\mathbb R^2$ is a map (or operator) $\sigma: D(\sigma) \subset \mathbb R \to \mathcal P (\mathbb R) \setminus \{ \emptyset \}$ such that
    \begin{equation}
    (\xi_1 - \xi_2) (x_1 - x_2) \ge 0, \qquad \forall x_i \in D(\sigma), \forall \xi_i \in \sigma(x_i).
    \end{equation}
    The set $D(\sigma)$ is called domain of the multivalued operator $\sigma$. Some authors define maximal monotone graphs as maps $\sigma:  \mathbb R \to \mathcal P (\mathbb R)$ and define $D(\sigma) = \{ x \in \Omega : \sigma(x) \ne \emptyset \}$.

    A monotone graph $\sigma$ is extended by another monotone graph $\tilde \sigma$ if $D(\sigma) \subset D(\tilde \sigma)$ and $\sigma(x) \subset \tilde \sigma(x)$ for all $x \in D (\sigma)$. A monotone graph is called maximal if its admits no proper extension. For further reference see \cite{Brezis:1973}.

\paragraph{Definition of solution}

The solution $u_\varepsilon$ is also well defined, although the set $K_\varepsilon$ must now be written as
\begin{equation*}
K_\varepsilon = \{ v \in H^1 (\Omega, \Gamma_2) : \forall x^{\prime }\in
G_\varepsilon , v (x^{\prime }) \in D(\sigma) \}.
\end{equation*}
We will have that the integral condition \eqref{integral inequality} turns into
\begin{equation}
\int\limits_{\Omega}\nabla\varphi\nabla(\varphi - u_\varepsilon)dx +
\varepsilon^{-k}\int\limits_{G_\varepsilon}\xi (x)(\varphi -
u_\varepsilon)dx^{\prime }\ge \int\limits_{\Omega}f(\varphi -
u_\varepsilon)dx,
\end{equation}
for all $\varphi \in K_\varepsilon$ and $\xi \in L^2 (G_\varepsilon)$ such
that $\xi(x^{\prime }) \in \sigma(\varphi(x^{\prime }))$ for a.e. $x^{\prime}\in G_\varepsilon$. Existence and uniqueness of this solutions follows as
in \cite{DiGoPoSh Nonlin Anal} and the references therein.

\paragraph{The auxiliary functions}

The equation of $\hat{w}$ is well-defined when $\sigma $ is a maximal monotone
graph. As we have proved in this paper, the estimate $0\leq H^{\prime
}\leq \lambda _{G_{0}}$ is independent of $\sigma $, and so $H$ is Lipschitz
continuous for any maximal monotone graph $\sigma $.

\paragraph{Signorini boundary conditions}

This is the case under study in this paper. Nonetheless, let us study in the
general setting. For this kind of boundary condition, we need to consider
the following maximal monotone graph:
\begin{equation}
	\label{eq:sigma Signorini and nonlinear}
\tilde{\sigma}(s)=\begin{dcases} \sigma (s) & s > 0 , \\ (-\infty,0] & s = 0, \\
\emptyset & s < 0
\end{dcases}
\end{equation}%
and $D(\sigma )=[0,+\infty )$. Let us compute $\tilde{H}_{G_{0}}$ in this
setting:

\begin{itemize}
\item For $u<0$, we can see what happens explicitly. We have that $u\leq
\hat{w}(u,\cdot )\leq 0$. Thus $u-\hat{w}\leq 0$. Since $D(\sigma
)=[0,+\infty )$ we must have that $\hat{w}(y;u)=u$ on $G_{0}$. But then $%
\hat{w}(y;u)=u\hat{\kappa}(y)$. Hence $\widetilde{H}_{G_{0}}(u)=\lambda
_{G_{0}}u$ when $u<0$.

\item When $u>0$, we have that $0\leq u-\hat{w}(u,\cdot )$. Thus, only the
values of $\sigma $ affect $H_{G_{0}}(u)$.
\end{itemize}

We conclude that
\begin{equation*}
	\widetilde{H}_{G_{0}}(u)=%
	\begin{cases}
	H_{G_{0}}(u) & u>0, \\
	\lambda _{G_{0}}u & u\leq 0.%
	\end{cases}%
\end{equation*}%
The computations with maximal monotone graphs yield precisely Theorem \ref%
{thm:1}. Notice that the bound on $\tilde{H}^{\prime }$ given by
\eqref{eq:H
prime bounds} is sharp.

\paragraph{Dirichlet boundary conditions}

In this case, we would have $D(\sigma )=\{0\}$ and
\begin{equation*}
\tilde{\sigma}(0)=(-\infty ,+\infty )
\end{equation*}%
By the same reasoning, we have that
\begin{equation*}
\tilde{H}_{G_{0}}(u)=\lambda _{G_{0}}u,
\end{equation*}%
for all $u\in \mathbb{R}$. In this case of Dirichlet boundary conditions the
critical case generates a linear term in the homogenized equation. This type
of phenomena was already notice by the authors of \cite{Cioranescu+Murat:1997}.

\paragraph{Cases of finite and infinite permeable coefficient}
The Signorini boundary condition imposed as a maximal monotone graph \eqref{eq:sigma Signorini and nonlinear} is the extreme case of \emph{infinite permeability}, aiming to represent the behaviour of very large \emph{finite} permeability given by a reaction term of the form
\begin{equation}
	\widetilde \sigma_ \mu (u) = \begin{dcases}
		\sigma(u) & u > 0, \\
		\mu u & u \le 0,
	\end{dcases}
\end{equation}
where $\mu$ is very large. As in Remark \ref{rem:linear sigma implies linear w and H} it is easy to show that the corresponding kinetic will be of the form
\begin{equation}
	\widetilde H_\mu (u) = \begin{dcases}
	H (u) & u > 0, \\
	\lambda_\mu u & u \le 0.
	\end{dcases}
\end{equation}
Furthermore, since we have proven that the Signorini boundary condition is an extremal case (i.e. $\widetilde H'(u) \le \lambda_{G_0}$) we have that $\lambda_\mu \le \lambda_{G_0}$.

\subsection{On the super-linearity condition}

The condition
\begin{equation*}
|\sigma(s) - \sigma(t)| \ge k_1 |s-t|
\end{equation*}
is only used in the proof of Lemma \ref{lem:trace theorem vee}. However, it
is our belief that this condition can be remove and still obtain the result.
We provide here a proof for $n=3$ and $G_0$ a ball.

We define the auxiliary function $w_{\varepsilon }$ unique solution of
\begin{equation}
\begin{dcases} -\Delta w_\ee^j = 0 & T_\ee^0 , \\ w_\ee = 0
& G_\ee^0\\ w_\ee = 1 & \partial T_\ee^0 \end{dcases}
\end{equation}%
where $T_{\varepsilon }^{0}$ is given by
\begin{equation}
T_{\varepsilon }^{0}=\left\{ x\in \mathbb{R}^{3}:\frac{x_{1}^2}{%
1-(a_{\varepsilon }\varepsilon ^{-1})^{2}}+x_{2}^{2}+x_{3}^{2}<\varepsilon
^{2}\right\} .
\end{equation}%
Using prolate ellipsoidal coordinates we can give an explicit expression of $%
w_{\varepsilon }$. These coordinates are given by
\begin{align}
x_{1}& =a_{\varepsilon }\sinh \psi \cos \theta _{1}, \\
x_{2}& =a_{\varepsilon }\cosh \psi \sin \theta _{1}\cos \theta _{2}, \\
x_{3}& =a_{\varepsilon }\cosh \psi \sin \theta _{1}\sin \theta _{2},
\end{align}%
where $0\leq \psi <\infty $, $0\leq \theta _{1}\leq \pi $ and $0\leq \theta
_{2}\leq 2\pi $. Defining $\sigma =\sinh \psi $ it can be proven through
symmetry that $w_{\varepsilon }(x)=V_{\varepsilon }(\sigma )$. Furthermore, $%
V_{\varepsilon }$ is the unique solution of the one-dimensional problem
\begin{equation}
\begin{dcases} \frac{d}{d\sigma} \left( (1 + \sigma^2) \frac{dV}{d\sigma}
\right) = 0 & \sigma \in \left( 0, \sqrt{ (a_\ee^{-1} \ee)^2 - 1} \right),
\\ V(0) = 0 , \\ V\left( \sqrt{ (a_\ee^{-1} \ee)^2 - 1} \right) = 1.
\end{dcases}
\end{equation}%
By integrating this simple one dimensional boundary value problem
\begin{equation}
V(\sigma )=\frac{\arctan \sigma }{\arctan \left( \sinh \sqrt{(a_{\varepsilon
}^{-1}\varepsilon )^{2}-1}\right) }.
\end{equation}%
Since we can recover from the change in variable
\begin{equation}
\sigma =\sinh \psi =\sqrt{\frac{|x|^{2}-a_{\varepsilon }^{2}+\sqrt{%
(a_{\varepsilon }^{2}-|x|^{2})^{2}+4x_{1}^{2}a_{\varepsilon }^{2}}}{%
2a_{\varepsilon }^{2}}}.
\end{equation}%
Due to mirror symmetry it is clear that $\partial _{x_{1}}w_{\varepsilon
}|_{\{x_{1}=0\}}=0$. Thus, we have
\begin{equation}
\int_{(T_{\varepsilon }^{0})^{+}}\nabla w_{\varepsilon }\nabla
(u^{2})dx=\int_{(\partial T_{\varepsilon }^{0})^{+}}\partial_{x_1} w_\ee u^2 dS -\int_{G_{\varepsilon
}^{0}}\partial _{x_{1}}w_{\varepsilon }u^{2}dS.
\end{equation}%
Using the explicit expression of $w_{\varepsilon }$ we can compute that
\begin{align}
\partial _{\nu }w_{\varepsilon }|_{(\partial T_{\varepsilon }^{0})^{+}}&
\sim { a_\ee \ee^{-2} } \\
\partial _{\nu }w_{\varepsilon }|_{G_{\varepsilon }^{0}}& \sim -{\ \frac{1}{%
\sqrt{a_{\varepsilon }^{2}-|x|^{2}}}}.
\end{align}

Now let
\begin{align}
T_\varepsilon^j &= P_\varepsilon^j + T_\varepsilon^0, \\
T_\varepsilon &= \bigcup_{ j \in \Upsilon_\varepsilon } T_\varepsilon^j, \\
W_\varepsilon (x) & = w_\varepsilon (x - P_\varepsilon^j) \quad \text{ for }
x \in T_\varepsilon^j.
\end{align}
Adding over $\Upsilon_\varepsilon$ we deduce that
\begin{equation}  \label{eq:superlinearity (5)}
\int_{ (T_\varepsilon)^+ } \nabla W_\varepsilon \nabla (u^2) ds = \sum_{j
\in \Upsilon_\varepsilon} \int_{(\partial T_\varepsilon^j)^+} \partial_\nu
w_\varepsilon^j u^2 ds - \int_{G_\varepsilon} \partial_{x_1} W_\varepsilon
u^2 ds.
\end{equation}
It is easy to prove that
\begin{equation}  \label{eq:superlinearity (6)}
\int_{(\partial T_\varepsilon^j)^+} \partial_\nu w_\varepsilon^j h^2 ds \le
K \sum_{j \in \Upsilon_\varepsilon} \| h \|^2_{H^1((T_\varepsilon^j)^+)}
\end{equation}
for any $h \in H^1 (\Omega)$. We now apply that
\begin{equation}  \label{eq:superlinearity (7)}
\| u \|_{L^2 (G_\varepsilon)}^2 \le K \left( \varepsilon^{-1} \| u \|_{L^2
(T_\varepsilon^+)} + \varepsilon \| \nabla u \|^2 _{L^2 (T_\varepsilon^+)}
\right).
\end{equation}
With this \eqref{eq:superlinearity (5)}, \eqref{eq:superlinearity (6)}, %
\eqref{eq:superlinearity (7)} we can prove Lemma \ref{lem:trace theorem vee}
for $k_1 = 0$.

\subsection{Connections to fractional operators}
	Let us consider a domain $\Omega =\Omega ^{\prime }\times (0,+\infty )$ where $\Omega' \subset \mathbb R^{n-1}$ is a smooth bounded domain.
	Then $\Gamma _{1}=\Omega ^{\prime }$ and $\Gamma _{2}=\partial \Omega \times
	(0,+\infty )$. The related problem
	\begin{equation}
	\begin{dcases}
	-\Delta u_{\varepsilon }=0, & (x,y)\in \Omega' \times (0,+\infty ), \\
	\partial _{\nu }u_{\varepsilon }+\varepsilon ^{-k}\sigma (u_{\varepsilon
	})=\varepsilon ^{-k}g_{\varepsilon }, & x\in G_{\varepsilon } \\
	\partial _{\nu }u_{\varepsilon }=0, & x\in \Omega ^{\prime }\setminus
	\overline{G}_{\varepsilon }, \\
	u_{\varepsilon }=0, & (x,y)\in \partial \Omega ^{\prime }\times (0,+\infty ),\\
	u_\ee \to 0 & |y| \to + \infty.
	\end{dcases}%
	\label{modified problem}
	\end{equation}%
	is very relevant because it can be linked with the study of the fractional
	Laplacian $(-\Delta )^{\frac{1}{2}}$. In fact, the boundary conditions on $%
	\Omega ^{\prime }$ can be written compactly as
	\begin{equation}
	\partial _{\nu }u_{\varepsilon }+\varepsilon ^{-k}\chi _{G_{\varepsilon
		}}\sigma (u_{\varepsilon })=\varepsilon ^{-k}\chi _{G_{\varepsilon
	}}g_{\varepsilon }\qquad x\in \Omega ^{\prime }
	\label{eq:normal derivative in Omega prime}
	\end{equation}%
	where $\chi $ is the indicator function. This boundary condition can be
	written as an equation of $\Omega ^{\prime }$ not involving the interior
	part of the domain, $\Omega ^{\prime }\times (0,+\infty )$, by understanding
	the normal derivative of problem \eqref{modified problem} as the fractional
	Laplace operator $(-\Delta )^{\frac{1}{2}}$ in $\Omega ^{\prime }$ (see \cite%
	{Caffarelli-Mellet}, \cite{DGCV} and their references). Then %
	\eqref{eq:normal derivative in Omega prime} can be written as
	\begin{equation}
	(-\Delta )^{\frac{1}{2}}u_{\varepsilon }+\varepsilon ^{-k}\chi
	_{G_{\varepsilon }}\sigma (u_{\varepsilon })=\varepsilon ^{-k}\chi
	_{G_{\varepsilon }}g_{\varepsilon }\qquad x\in \Omega ^{\prime }
	\label{eq:fractional laplace equation}
	\end{equation}%
	Thus, the study of the limit of \eqref{modified problem} will provide an
	homogenization result for \eqref{eq:fractional laplace equation}. Applying
	similar techniques to this paper and previous results in the literature \cite%
	{DiGoPoSh Math Anal}, the homogenized problem
	\begin{equation}
	(-\Delta )^{\frac{1}{2}}u_{0}+CH(x,u_{0})=Ch\qquad x\in \Omega ^{\prime }
	\end{equation}%
	is expected, where $H$ and $h$ will depend on $\sigma $ and $g_{\varepsilon
	} $. This could provide some new results of critical size homogenization for
	the fractional Laplacian (in the spirit of the important work \cite%
	{Caffarelli-Mellet}, where some random aspects on the net, and for a general
	fractional power of the Laplacian, are also considered).

\part{The case $n =2$}

\section{ Proof of Theorem \protect\ref{thm:2}}

It's well known that problem \eqref{basic set n=2},
\eqref{integral identity
n=2} has a unique weak solution $u_{\varepsilon}\in H^{1}(\Omega,\Gamma_2)$.
By using \eqref{integral identity n=2} and conditions \eqref{sigma cond},
that was set on the function $\sigma$, we get the following estimates
\begin{equation}  \label{u epsilon estim n=2}
\Vert\nabla u_\varepsilon\Vert_{L_2(\Omega)} \le K, \qquad e^{\frac{\alpha^2%
}{\varepsilon}}\Vert u_\varepsilon\Vert^2_{L_2(G_\varepsilon)}\le K_1,
\end{equation}
here and below, constants $K,\, K_1$ are independent of $\varepsilon$.

Hence there exists subsequence (denote as the original sequence $%
u_\varepsilon$) such that as $\varepsilon\rightarrow0$ we have
\begin{equation}  \label{u0 def n=2}
\begin{split}
u_\varepsilon \rightharpoonup u_0 \mbox{ weakly in } H^1_0(\Omega), \\
u_\varepsilon \rightarrow u_0 \mbox{ strongly in } L_2(\Omega).
\end{split}%
\end{equation}
We introduce auxiliary functions $w^{j}_{\varepsilon}$ and $%
q^{j}_{\varepsilon}$ as a weak solution to the following problems
\begin{equation}  \label{prob w n=2}
\left\{%
\begin{array}{lr}
\Delta{w^{j}_{\varepsilon}}=0, & x\in T^{j}_{{\varepsilon}/4}\setminus
\overline{T^{j}_{a_{\varepsilon}}}, \\
w^{j}_{\varepsilon}=1, & x\in \partial{T^{j}_{a_{\varepsilon}}}, \\
w^{j}_{\varepsilon}=0, & x\in \partial{T^{j}_{{\varepsilon}/4}},%
\end{array}%
\right.
\end{equation}
and
\begin{equation}  \label{prob q n=2}
\left\{%
\begin{array}{lr}
\Delta{q^{j}_{\varepsilon}}=0, & x\in T^{j}_{\varepsilon}\setminus \overline{%
l^{j}_{\varepsilon}}, \\
q^{j}_{\varepsilon}=1, & x\in l^{j}_{\varepsilon}, \\
q^{j}_{\varepsilon}=0, & x\in \partial{T^{j}_{{\varepsilon}/4}}.%
\end{array}%
\right.
\end{equation}

\begin{figure}[h]
\centering
\includegraphics[width=.5\textwidth]{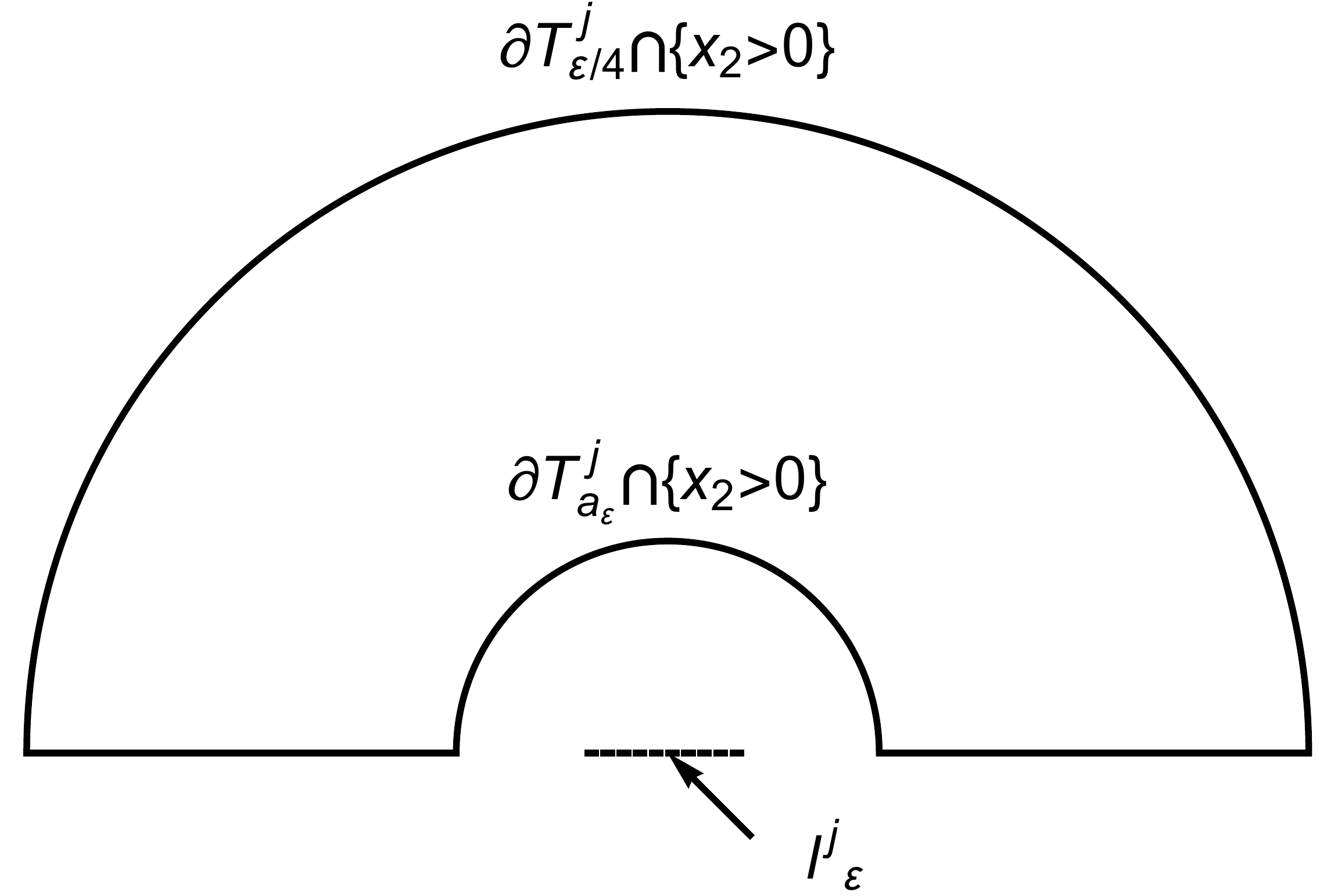}
\caption{Domain $T^{+,j}_\protect\varepsilon\setminus\overline{T^j_{a_%
\protect\varepsilon}}$ and $l^j_\protect\varepsilon$.}
\end{figure}
Note that $w^{j}_{\varepsilon}$ and $q^{j}_{\varepsilon}$ are also a
solutions of the boundary value problems in the domains $(T^{j}_{{\varepsilon%
}/4})^{+}\setminus\overline{T^{j}_{a_{\varepsilon}}}$ and $(T^{j}_{{%
\varepsilon}/4})^{+}$ respectively, where $(T^{j}_{r})^{+}=T^{j}_{r}\cap%
\{x_{2}>0\}$,

\begin{align}  \label{prob w var2 n=2}
		&
	\begin{dcases}
	\Delta{w^{j}_{\varepsilon}}=0, & x\in (T^{j}_{{\varepsilon}/4})^{+}\setminus
	\overline{T^{j}_{a_{\varepsilon}}}, \\
	w^{j}_{\varepsilon}=0, & x\in \partial{T^{j}_{{\varepsilon}/4}}\cap
	\{x_{2}>0\}, \\
	w^{j}_{\varepsilon}=1, & x\in \partial{T^{j}_{a_{\varepsilon}}}%
	\cap\{x_{2}>0\}, \\
	\partial_{x_{2}}w^{j}_{\varepsilon}=0, & x\in \{x_{2}=0\}\cap(\partial{%
	T^{j}_{{\varepsilon}/4}\setminus\overline{T^{j}_{a_{\varepsilon}}}}),%
	\end{dcases}%
	\\
	 \label{prob q var2 n=2}
	&%
	\begin{dcases}
	\Delta{q^{j}_{\varepsilon}}=0, & x\in (T^{j}_{{\varepsilon}/4})^+, \\
	q^j_\ee = 0 & x \in (\partial T_{\ee / 4}^j)^+ \\
	q^{j}_{\varepsilon}=1, & x\in l^{j}_{\varepsilon}, \\
	\partial_{x_{2}}q^{j}_{\varepsilon}=0, & x\in (T^{j}_{{\varepsilon}%
	/4}\cap\{x_{2}=0\})\setminus\overline{l^{j}_{\varepsilon}},%
	\end{dcases}%
\end{align}
where $j\in \Upsilon_{\varepsilon}$, $l^{j}_{\varepsilon}=a_{\varepsilon}%
\hat{l}_{0}+{\varepsilon}j$.
Define
\begin{equation}  \label{prob W n=2}
W_{\varepsilon}(x)=\left\{%
\begin{array}{lr}
w^{j}_{\varepsilon}(x), & x\in (T^{j}_{{\varepsilon}/4})^{+}\setminus
\overline{T^{j}_{a_{\varepsilon}}}, j\in \Upsilon_{\varepsilon}, \\
1, & x\in (T^{j}_{a_{\varepsilon}})^{+}, \\
0, & x\in \mathbb{R}^{2}_{+}\setminus \bigcup_{j\in \Upsilon_{\varepsilon}}%
\overline{T^{j}_{{\varepsilon}/4}},%
\end{array}%
\right.
\end{equation}
where $\mathbb{R}^{2}_{+}=\{x_{2}>0\}$,

\begin{equation}  \label{prob Q n=2}
Q_{\varepsilon}(x)=\left\{%
\begin{array}{lr}
q^{j}_{\varepsilon}(x), & x\in (T^{j}_{{\varepsilon}/4})^{+}, \,j\in
\Upsilon_{\varepsilon}, \\
0, & x\in \mathbb{R}^{2}_{+}\setminus \bigcup_{j\in \Upsilon_{\varepsilon}}%
\overline{T^{j}_{{\varepsilon}/4}}^{+}.%
\end{array}%
\right.
\end{equation}
We have $W_{\varepsilon}, Q_{\varepsilon}\in H^{1}_{0}(\Omega)$ and
\begin{equation}  \label{W conv n=2}
W_{\varepsilon}\rightharpoonup 0, \,\,\,\,\mbox{weakly in}
\,\,\,H^{1}_{0}(\Omega),\,\,{\varepsilon}\to 0.
\end{equation}

\begin{lemma}
\label{lemma estim w-q n=2} Let $W_{\varepsilon}$ be a function defined by
the formula \eqref{prob W n=2}, $Q_{\varepsilon}$ be a function defined by
the formula \eqref{prob Q n=2}. Then
\begin{equation}  \label{W-Q estimation n=2}
\Vert{W_{\varepsilon}-Q_{\varepsilon}}\Vert_{H^{1}(\Omega)}\le K\sqrt{%
\varepsilon}.
\end{equation}
\end{lemma}

\begin{proof}
Note that for an arbitrary function $ \psi\in H^{1}(T^{j}_{{\ee}/4})$ such that $\psi=0$ on $l^{j}_{\ee}$ we have
\begin{equation}\label{W-Q estim int ident Q}
\int\limits_{(T^{j}_{{\ee}/4})^{+}}\nabla{q^{j}_{\ee}}\nabla\psi{dx_{1}dx_{2}}=0.
\end{equation}
We consider $\psi=w_\ee^{j}-q^{j}_{\ee}$ as a test function in the obtained equality and get
\begin{equation}\label{W-Q estim int ident Q test func}
\int\limits_{(T^{j}_{{\ee}/4})^{+}}\nabla{q^{j}_{\ee}}\nabla(w^{j}_{\ee}-q^{j}_{\ee})dx_{1}dx_{2}=0.
\end{equation}
In addition, we have
\begin{equation}\label{W-Q estim ident W}
\int\limits_{(T^{j}_{{\ee}/4})^{+}}\nabla{w^{j}_{\ee}}\nabla(w^{j}_{\ee}-q^{j}_{\ee})dx_{1}dx_{2}=
\int\limits_{\pa{T^{j}_{a_{\ee}}\cap\{x_{2}>0\}}}\pa_{\nu}w^{j}_{\ee}(w^{j}_{\ee}-q^{j}_{\ee})ds.
\end{equation}
By subtracting \eqref{W-Q estim int ident Q test func} from \eqref{W-Q estim ident W} we derive
\begin{equation}\label{W-Q estim diff}
\int\limits_{(T^{j}_{{\ee}/4})^{+}}|\nabla(w^{j}_{\ee}-q^{j}_{\ee})|^{2}dx=\int\limits_{\pa{T^{j}_{a_{\ee}}}\cap\{x_{2}>0\}}\pa_{\nu}w^{j}_{\ee}(w^{j}_{\ee}-q^{j}_{\ee})ds.
\end{equation}
Note that $w^{j}_{\ee}(x)=\frac{\ln(4r/{\ee})}{\ln(4a_{\ee}/{\ee})}$ and $\pa_{\nu}w^{j}_{\ee}=-\frac{1}{a_{\ee}\ln(4a_{\ee}/{\ee})}$. Hence, \eqref{W-Q estim diff} implies that
\begin{align*}
\Vert\nabla(w^{j}_{\ee}-q^{j}_{\ee})\Vert^{2}_{L^{2}((T^{j}_{{\ee}/4})^{+})} &\le \frac{1}{a_{\ee}|\ln(4a_{\ee}/{\ee})|}\int\limits_{\pa{T^{j}_{a_{\ee}}}\cap\{x_{2}>0\}}|w^{j}_{\ee}-q^{j}_{\ee}|ds\\
&=\frac{1}{|\ln(4a_{\ee}/{\ee})}|\int\limits_{\pa{T^{j}_{1}\cap\{y_{2}>0\}}}|w^{j}_{\ee}-q^{j}_{\ee}|ds_{y}\equiv J_{\ee}.
\end{align*}
Given that $w^{j}_{\ee}-q^{j}_{\ee}=0$ if $y\in \hat{l}_{0}$ and using the embedding theorem, we get
\begin{equation}\label{Q-W estim pre-final estim}
J_{\ee}\le \frac{K}{|\ln(4a_{\ee}/{\ee})|}\Bigl(\int\limits_{(T^{j}_{1})^{+}}|\nabla_{y}(w^{j}_{\ee}-q^{j}_{\ee})|^{2}dy\Bigr)^{1/2}\le K{\ee}\Vert\nabla(w^{j}_{\ee}-q^{j}_{\ee})\Vert_{L^{2}((T^{j}_{{\ee}/4})^{+})}.
\end{equation}
From here we derive the estimate
\begin{equation*}
\Vert\nabla(w^{j}_{\ee}-q^{j}_{\ee})\Vert_{L^{2}((T^{j}_{{\ee}/4})^{+})}\le K{\ee}.
\end{equation*}
From this estimation it follows that
\begin{equation}\label{Q-W estim final estim}
\Vert{W_{\ee}-Q_{\ee}}\Vert_{H^{1}(\Omega)}\le K\sqrt{\ee}.
\end{equation}
This concludes the proof. \end{proof}
We introduce function $m(y)\in H^{1}((T^{0}_{1})^{+})$ as the weak solution to
the following boundary value problem
\begin{equation}  \label{prob m n=2}
\left\{%
\begin{array}{lr}
\Delta_{y}m=0, & y\in T^{0}_{1}\cap\{y_{2}>0\}=(T^{0}_{1})^{+}, \\
\partial_{y_{2}}m=1, & y\in \hat{l}_{0}, \\
\partial_{\nu}m=\frac{2l_{0}}{\pi}, & y\in \partial{T^{0}_{1}}%
\cap\{y_{2}>0\}=(\partial{T^{0}_{1}})^{+}, \\
\partial_{y_{2}}m=0, & y\in \partial(T^{0}_{1})^{+}\setminus \hat{l}_{0}\cup
(\partial{T^{0}_{1}})^{+}.%
\end{array}%
\right.
\end{equation}
Consider
\begin{equation}  \label{m cell n=2}
m^{j}_{\varepsilon}(x)={\varepsilon}m(\frac{x-P^{j}_{\varepsilon}}{%
a_{\varepsilon}}), \,\,\, x\in (T^{j}_{a_{\varepsilon}})^{+}.
\end{equation}
The function $m^{j}_{\varepsilon}(x)$ verifies the problem
\begin{equation}  \label{m cell problem}
\left\{%
\begin{array}{lr}
\Delta_{x}m^{j}_{\varepsilon}=0, & x\in (T^{j}_{a_{\varepsilon}})^{+}, \\
\partial_{\nu}m^{j}_{\varepsilon}=\frac{{\varepsilon}a_{%
\varepsilon}^{-1}2l_{0}}{\pi}, & x\in (\partial{T^{j}_{a_{\varepsilon}}}%
)^{+}, \\
\partial_{x_{2}}m^{j}_{\varepsilon}={\varepsilon}a_{\varepsilon}^{-1}, &
x\in \{x_{2}=0, |x-P^{j}_{\varepsilon}|\le
a_{\varepsilon}l_{0}\}=l^{j}_{\varepsilon}, \\
\partial_{x_{2}}m^{j}_{\varepsilon}=0, & \mbox{on the rest of the boundary}.%
\end{array}%
\right.
\end{equation}

\begin{lemma}
\label{lemma m comparison} Let $n = 2$ and $h \in H^1(\Omega, \Gamma_2)$
then the following estimate holds:
\begin{equation}
\Bigl|\frac{2 l_0 \varepsilon}{\pi a_\varepsilon}\sum\limits_{j\in\Upsilon_%
\varepsilon}\int\limits_{(\partial{T^{j}_{a_{\varepsilon}}})^{+}}h ds -
\frac{\varepsilon}{a_\varepsilon}\int\limits_{l_{\varepsilon}}h dx_{1} %
\Bigr| \le K\sqrt{\varepsilon}.
\end{equation}
\end{lemma}

\begin{proof}
Denote $h_{\ee}=H(\psi)(\psi-H(\psi)-u_{\ee})$. Then
\begin{equation}\label{lemma m conv int}
\begin{gathered}
\Bigl|\frac{2l_{0}{\ee}a_{\ee}^{-1}}{\pi}\int\limits_{(\pa{T^{j}_{a_{\ee}}})^{+}}h_{\ee}ds-{\ee}a_{\ee}^{-1}\int\limits_{l^{j}_{\ee}}h_{\ee}dx_{1}\Bigr|=\\
=\Bigl|\int\limits_{(T^{j}_{a_{\ee}})^{+}}\nabla_{x}m^{j}_{\ee}\nabla{h_{\ee}}dx\Bigr|\le\Vert\nabla_{x}m^{j}_{\ee}\Vert_{L^{2}((T^{j}_{a_{\ee}})^{+})}\Vert\nabla{h_{\ee}}\Vert_{L^{2}((T^{j}_{a_{\ee}})^{+})}.\\
\end{gathered}
\end{equation}
Due to the fact that
\begin{align*}
\Vert\nabla_{x}m^{j}_{\ee}\Vert^{2}_{L^{2}((T^{j}_{a_{\ee}})^{+})}&={\ee}^{2}\Vert\nabla_{y}m(y)\Vert^{2}_{L^{2}((T^{0}_{1})^{+})}\le K{\ee}^{2},
\end{align*}
we have
\begin{equation}\label{m estim n=2}
\sum\limits_{j\in \Upsilon_{\ee}}\Vert\nabla_{x}m^{j}_{\ee}\Vert^{2}_{L^{2}((T^{j}_{a_{\ee}})^{+})}\le K{\ee}.
\end{equation}
From \eqref{lemma m conv int}, \eqref{m estim n=2} we derive
\begin{equation}\label{m int estim}
\begin{gathered}
\Bigl|e^{\alpha^{2}/{\ee}}\frac{\pi}{2l_{0}}\int\limits_{l_{\ee}}h_{\ee}dx_{1}-e^{\alpha^{2}/{\ee}}\sum\limits_{j\in \Upsilon_{\ee}}\int\limits_{\pa{T^{j}_{a_{\ee}}}\cap\{x_{2}>0\}}h_{\ee}ds\Bigr|\le\\
\le \delta^{-1}\sum\limits_{j\in \Upsilon_{\ee}}\Vert{\nabla_{x}m^{j}_{\ee}}\Vert^{2}_{L^{2}((T^{j}_{a_{\ee}})^{+})}+\delta
\Vert\nabla{h_{\ee}}\Vert^{2}_{L^{2}(\Omega)}\le K\sqrt{\ee},
\end{gathered}
\end{equation}
if $\delta=\sqrt{\ee}$.
\end{proof}

\begin{proof}[Proof of Theorem \ref{thm:2}]
First of all equation \eqref{func eq} has a unique solution $H(u)$ that is a Lipschitz continuous function in $\mathbb{R}$ and satisfies
\begin{equation}\label{H prop n=2}
\begin{gathered}
(H(u) - H(v))(u - v) \ge \tilde k_1 |u - v|,\\
|H(u)|\le |u|,
\end{gathered}
\end{equation}
for all $u,\, v\in\mathbb{R}$ and a certain constant $\tilde k_1 > 0$.

We take $v=\psi-Q_{\ee}(H(\psi^{+}) + \psi^{-})$ as a test function in \eqref{integral identity n=2}, where $\psi\in C^{\infty}(\overline\Omega)$, $\psi(x)=0$ in the neighborhood of $\Gamma_{2}$,
$H(u)$ satisfies the functional equation \eqref{func eq}. Note that from \eqref{prob q n=2}, \eqref{prob Q n=2} and \eqref{H prop n=2} we have $v\ge 0$ on $l_\ee$ so $v\in K_\ee$. Hence we get
\begin{equation}\label{hom theorem int ineq n=2}
\begin{gathered}
\int\limits_{\Omega}\nabla(\psi-Q_{\ee}(H(\psi^{+}) + \psi^{-}))\nabla(\psi-Q_{\ee}(H(\psi^{+}) + \psi^{-})-u_{\ee})dx+\\
+e^{\frac{\alpha^{2}}{\ee}}\int\limits_{l_{\ee}}\sigma(\psi^{+}-H(\psi^{+}))(\psi^{+}-H(\psi^{+})-u_{\ee})dx_{1}\ge\\
\ge\int\limits_{\Omega}f(\psi-Q_{\ee}(H(\psi^{+}) + \psi^{-})-u_{\ee})dx.
\end{gathered}
\end{equation}
We rewrite inequality \eqref{hom theorem int ineq n=2} in the following way
\begin{equation}\label{transform int ineq n=2}
\begin{gathered}
\int\limits_{\Omega}\nabla(\psi-W_{\ee}(H(\psi^{+}) + \psi^{-} ))\nabla(\psi-Q_{\ee}(H(\psi^{+}) + \psi^{-})-u_{\ee})dx-\\
-\int\limits_{\Omega}\nabla((Q_{\ee}-W_{\ee})(H(\psi^{+}) + \psi^{-}))\nabla(\psi-Q_{\ee}(H(\psi^{+}) + \psi^{-})-u_{\ee})dx+\\
+e^{\frac{\alpha^{2}}{\ee}}\int\limits_{l_{\ee}}\sigma(\psi^{+}-H(\psi^{+}))(\psi^{+}-H(\psi^{+})-u_{\ee})dx_{1}\ge\\
\ge \int\limits_{\Omega}f(\psi-Q_{\ee}(H(\psi^{+}) + \psi^{-})-u_{\ee})dx.
\end{gathered}
\end{equation}
From the fact, that $Q_{\ee}\rightharpoonup 0$ as ${\ee}\to 0$ weakly in $H^{1}(\Omega, \Gamma_{2})$, we have
\begin{equation}\label{f lim n=2}
\lim\limits_{{\ee}\to 0}\int\limits_{\Omega}f(\psi-Q_{\ee}(H(\psi^{+}) + \psi^{-})-u_{\ee})dx=\int\limits_{\Omega}f(\psi-u_{0})dx,
\end{equation}
\begin{equation}\label{psi lim n=2}
\lim\limits_{{\ee}\to 0}\nabla\psi\nabla(\psi-Q_{\ee}(H(\psi^{+}) + \psi^{-})-u_{\ee})dx=\int\limits_{{\ee}\to 0}\nabla\psi\nabla(\psi-u_{0})dx.
\end{equation}
Lemma~\ref{lemma estim w-q n=2} implies that
\begin{equation}\label{Q-W hom theorem estim n=2}
\lim\limits_{{\ee}\to 0}\int\limits_{\Omega}\nabla(Q_{\ee}-W_{\ee}(H(\psi^{+}) + \psi^{-}))\nabla(\psi-Q_{\ee}(H(\psi^{+}) + \psi^{-})-u_{\ee})dx=0.
\end{equation}
Consider the remaining integrals in \eqref{transform int ineq n=2}. Denote
\begin{align*}
I_{\ee}&\equiv -\int\limits_{\Omega}\nabla(W_{\ee}(H(\psi^{+}) + \psi^{-}))\nabla(\psi-Q_{\ee}(H(\psi^{+}) + \psi^{-})-u_{\ee})dx=\\
&=-\int\limits_{\Omega}\nabla{W_{\ee}}\nabla\{(H(\psi^{+}) + \psi^{-})(\psi-Q_{\ee}(H(\psi^{+}) + \psi^{-})-u_{\ee})\}dx+\alpha_{\ee},
\end{align*}
where $\alpha_{\ee}\to 0$ as ${\ee}\to 0$.
\\
It is easy to see that
\begin{align}
I_{\ee}&=-\int\limits_{\Omega}\nabla{W_{\ee}}\nabla\{(H(\psi^{+}) + \psi^{-})(\psi-Q_{\ee}(H(\psi^{+}) + \psi^{-})-u_{\ee})\}dx \nonumber\\
&=-\sum\limits_{j\in \Upsilon_{\ee}}\int\limits_{(T^{j}_{{\ee}/4})^{+}\setminus\overline{T^{j}_{a_{\ee}}}}\nabla{w^{j}_{\ee}}\nabla\{(H(\psi^{+}) + \psi^{-})(\psi-q^{j}_{\ee}(H(\psi^{+}) + \psi^{-})-u_{\ee})\}dx+\widetilde{\alpha_{\ee}}\nonumber\\
&=-\sum\limits_{j\in \Upsilon_{\ee}}\int\limits_{\pa{T^{j}_{{\ee}/4}}\cap\{x_{2}>0\}}\pa_{\nu}w^{j}_{\ee}(H(\psi^{+}) + \psi^{-})(\psi-u_{\ee})ds \nonumber\\
	\label{W int decomposition n=2}
&\qquad -\sum\limits_{j\in \Upsilon_{\ee}}\int\limits_{\pa{T^{j}_{a_{\ee}}}\cap\{x_{2}>0\}}\pa_{\nu}w^{j}_{\ee}(H(\psi^{+}) + \psi^{-})(\psi^{+}-H(\psi^{+})-u_{\ee})ds+\widetilde{\alpha_{\ee}},
\end{align}
where $\widetilde{\alpha_{\ee}}\to 0$, ${\ee}\to 0$.
\\
Since $\pa_{\nu}w^{j}_{\ee}\Bigl|_{\pa{T^{j}_{{\ee}/4}}}=\frac{4}{{\ee}\ln(4a_{\ee}/{\ee})}=\frac{4}{-\alpha^{2}+{\ee}\ln(4C_{0})}$,  using the results of \cite{LoOlPeSh}, we derive
\begin{align}
&-\lim\limits_{{\ee}\to 0}\sum\limits_{j\in \Upsilon_{\ee}}\int\limits_{\pa{T^{j}_{{\ee}/4}}\cap\{x_{2}>0\}}\pa_{\nu}w^{j}_{\ee}(H(\psi^{+}) + \psi^{-})(\psi-u_{\ee})ds\nonumber\\
&\qquad =\lim\limits_{{\ee}\to 0}\frac{4}{\alpha^{2}-{\ee}\ln(4C_{0})}\sum\limits_{j\in \Upsilon_{\ee}}\int\limits_{\pa{T^{j}_{{\ee}/4}}\cap\{x_{2}>0\}}(H(\psi^{+}) + \psi^{-})(\psi-u_{\ee})ds \nonumber \\
&\qquad=\frac{\pi}{\alpha^{2}}\int\limits_{\Gamma_{2}}(H(\psi^{+}) + \psi^{-})(\psi-u_{0})ds.
\label{lim ball n=2}
\end{align}
Let us find the limit of the expression
\begin{align}
&-\sum\limits_{j\in \Upsilon_{\ee}}\int\limits_{\pa{T^{j}_{a_{\ee}}}\cap\{x_{2}>0\}}\pa_{\nu}w^{j}_{\ee}(H(\psi^{+}) + \psi^{-})(\psi^{+}-H(\psi^{+})-u_{\ee})ds \nonumber\\
&\qquad \qquad  +e^{\frac{\alpha^{2}}{\ee}}\int\limits_{l_{\ee}}\sigma(\psi^{+}-H(\psi^{+}))(\psi^{+}-H(\psi^{+})-u_{\ee})dx_{1}\nonumber\\
&\qquad =\sum\limits_{j\in \Upsilon_{\ee}}\frac{(\alpha^{2}C_{0})^{-1}e^{\alpha^{2}/{\ee}}}{1-{\ee}\alpha^{-2}\ln(4C_{0})}\int\limits_{\pa{T^{j}_{a_{\ee}}}\cap\{x_{2}>0\}}(H(\psi^{+}) + \psi^{-})(\psi^{+}-H(\psi^{+})-u_{\ee})ds\nonumber\\
&\qquad \qquad +e^{\alpha^{2}/{\ee}}\int\limits_{l_{\ee}}\sigma(\psi^{+}-H(\psi^{+}))(\psi^{+}-H(\psi^{+})-u_{\ee})dx_{1}\nonumber\\
&\quad =e^{\alpha^{2}/{\ee}}\int\limits_{l_{\ee}}\sigma(\psi^{+}-H(\psi^{+}))(\psi^{+}-H(\psi^{+})-u_{\ee})dx_{1}\nonumber\\
&\qquad \qquad -\frac{e^{\alpha^{2}/{\ee}}}{\alpha^{2}C_{0}}\sum\limits_{j\in \Upsilon_{\ee}}\int\limits_{\pa{T^{j}_{a_{\ee}}}\cap\{x_{2}>0\}}(H(\psi^{+}) + \psi^{-})(\psi^{+}-H(\psi^{+})-u_{\ee})ds +\hat{\alpha}_{\ee} \nonumber\\
&\quad \equiv D_{\ee}+\hat{\alpha}_{\ee},
\label{expression limit n=2}
\end{align}
where $\hat{\alpha}_{\ee}\to 0$, ${\ee}\to 0$.

To conclude the proof we will estimate the limit of $D_\ee$. We have
\begin{equation}\label{D eps estim n=2}
\begin{gathered}
D_{\ee} = \left\{\frac{\pi}{2l_{0}\alpha^{2}C_{0}}e^{\alpha^{2}/{\ee}}\int\limits_{l_{\ee}}(H(\psi^{+}) + \psi^{-})(\psi^{+}-H(\psi)-u_{\ee})dx_{1}\right.-\\
-\left.\frac{e^{{\alpha^{2}}/{\ee}}}{\alpha^{2}C_{0}}\sum\limits_{j\in \Upsilon_{\ee}}\int\limits_{\pa{T^{j}_{a_{\ee}}}\cap\{x_{2}>0\}}(H(\psi^{+}) + \psi^{-})(\psi^{+}-H(\psi^{+})-u_{\ee})ds\right\}+\\
+e^{\alpha^{2}/{\ee}}\int\limits_{l_{\ee}}\{\sigma(\psi^{+}-H(\psi^{+}))-\frac
{\pi}{2l_{0}\alpha^{2}C_{0}}H(\psi^{+})\}(\psi^{+}-H(\psi^{+})-u_{\ee})dx_{1} - \\
- \frac{\pi}{2l_{0}\alpha^{2}C_{0}}e^{\alpha^{2}/{\ee}}\int\limits_{l_{\ee}}\psi^{-}(\psi^{+}-H(\psi)-u_{\ee})dx_{1} = \mathcal J^1_\ee + \mathcal J^2_\ee + \mathcal J^3_\ee
\end{gathered}
\end{equation}
Lemma~\ref{lemma m comparison} implies that
\begin{equation}\label{J1 lim n=2}
|\mathcal J^1_\ee| \le K\sqrt{\ee}.
\end{equation}
Then $\mathcal J^2_\ee$ vanishes due to equation~\eqref{func eq}. By using that $u_\ee\ge0$, $\psi^{-}\le 0$ on $l_\ee$ and the fact that $\psi^{-}(\psi^{+} - H(\psi^{+}))\equiv 0$ we have
\begin{equation}\label{J3 lim n=2}
\mathcal J^3_\ee \le 0.
\end{equation}
Hence, we have that $\lim\limits_{\ee\rightarrow0}D_\ee \le 0$ and
\begin{equation}\label{I lim estim n=2}
\begin{gathered}
\lim\limits_{\ee\rightarrow 0}\left(e^{\frac{\alpha^{2}}{\ee}}\int\limits_{l_{\ee}}\sigma(\psi^{+}-H(\psi^{+}))(\psi^{+}-H(\psi^{+})-u_{\ee})dx_{1} + I_\ee\right)\le\\
\le \frac{\pi}{\alpha^{2}}\int\limits_{\Gamma_{2}}(H(\psi^{+}) + \psi^{-})(\psi-u_{0})ds.
\end{gathered}
\end{equation}
Therefore, from \eqref{hom theorem int ineq n=2}-\eqref{I lim estim n=2} we conclude that $u_{0}\in H^{1}(\Omega, \Gamma_{2})$ satisfies the following inequality
\begin{equation}\label{last int ident n=2}
\int\limits_{\Omega}\nabla\psi\nabla(\psi-u_{0})dx+\frac{\pi}{\alpha^{2}}\int\limits_{\Gamma_{1}}(H(\psi^{+}) + \psi^{-})(\psi-u_{0})dx_{1}\ge
\int\limits_{\Omega}f(\psi-u_{0})dx,
\end{equation}
for any $\psi \in H^{1}(\Omega, \Gamma_{2})$, where $H(u)$ satisfies the functional equation \eqref{func eq}.
This concludes the proof.
\end{proof}

\section*{Acknowledgments}

The research of J.I. D\'{\i}az and D. G%
\'{o}mez-Castro was partially supported by the projects ref. MTM
2014-57113-P and MTM2017-85449-P of the DGISPI (Spain).


\begin{thebibliography}{99}
\bibitem{Borisov} D.I. Borisov.  On a model boundary value problem for
Laplacian with frequently alternating type of boundary condition.
Asymptot. anal., v.35, N.1, 2003, p.~1-26.

\bibitem{Brezis:1973} H. Brezis.Operateurs Maximaux Monotones et
Semi-groupes de Contractions dans les Espaces de Hilbert. North-Holland,
Amsterdam (1973)

\bibitem{Caffarelli-Mellet} L.A. Caffarelli and A. Mellet, Random
homogenization of an obstacle problem, Ann. I. H. Poincar\'{e} -- AN 26
(2009) 375--395.

\bibitem{Chae} H.-R. Chae, J. Lee, Ch-H. Lee, In-Ch. Kim and P.-K. Park,
Graphene oxide-embedded thin-film composite reverse osmosis

membrane with high flux, anti-biofouling,and chlorine resistance, Journal of
Membrane Science 483 (2015) 128--135.

\bibitem{Chech93} G.A. Chechkin. Avearging of boundary value problemns with a
singular perturbation of the boundary conditions. Sbornik Mathematics,
v.79, N.1, 1994, p.~191-222.

\bibitem{Cioranescu+Murat:1997} D.~Cioranescu and F.~Murat.
\newblock {A
Strange Term Coming from Nowhere}. \newblock In A.~Cherkaev and R.~Kohn,
editors, \emph{Topics in Mathematical Modelling of Composite Materials},
pages 45--94. Springer, New York, 1997.

\bibitem{Conca+Donato} C. Conca, P. Donato: Non-homogeneous Neumann
problems in domains with small holes. Mod\'{e}lisation Math\'{e}matique
Anal. Num\'{e}rique. 22, 561--607 (1988)

\bibitem{Damlamian} A. Damlamian, Li Ta-Tsien, Boundary homogenization for
elliptic problems. J. Math. Pures Appl. (9), 66:4 (1987), 351-361.

\bibitem{DiGoPoSh16} J.I. D\'{\i}az, D. Gomez-Castro, A.V. Podol'skii, T.A.
Shaposhnikova. Homogenization of the p-Laplace operator with nonlinear
boundary condition on critical size particles:identifying the strange term
for some non smooth and multivalued operators. Doklady Mathematics, 94 \ 1
(2016) 387-392.

\bibitem{DiGoPoSh17} J.I. D\'{\i}az, D. Gomez-Castro, A.V. Podol'skii, T.A.
Shaposhnikova. Homogenization of variational inequalities of Signorini type
for the p-Laplacian in perforated domains when $p\in (1,2)$. Doklady
Mathematics, 2017, 95(2):151-156.

\bibitem{DiGoPoSh Nonlin Anal} J.I. D\'{\i}az, D. Gomez-Castro, A.V.
Podol'skii, T.A. Shaposhnikova.Characterizing the strange term in critical
size homogenization: quasilinear equations with a nonlinear boundary
condition involving a general maximal monotone graph. Advances in
Nonlinear Analysis. 2017. DOI:10.1515/anona-2017-0140.

\bibitem{DiGoPoSh Math Anal} J.I. D\'{\i}az, D. Gomez-Castro, A.V.
Podolskii, T.A. Shaposhnikova. On the asymptotic limit of the effectiveness
of reaction-diffusion equations in perforated media. J. Math. Anal. Appl.
455(2017), 1597-1613. 

\bibitem{DiGoPoSh Doklady 2018} J.I. D\'{\i}az, D. Gomez-Castro, A.V.
Podolskii, T.A. Shaposhnikova, Homogenization of Boundary Value Problems in
Plane Domains with Frequently Alternating Type of Nonlinear Boundary
Conditions: Critical Case, Doklady Mathematics, 97 3 (2018) 271--276.
Published in Russian in Doklady Akademii Nauk, 480 6 (2018) 644--649.

\bibitem{DiGoShZu Aux Shape} J.I. D\'{\i}az, D. Gomez-Castro, T.A.
Shaposhnikova, M.N. Zubova. Change of homogenized absorption term in
diffusion processes with reaction on the boundary of periodically
distributed asymmetric particles of critical size. Electronic Journal of
Differential Equations, V.2017 (2017) 178, 1-25.

\bibitem{DGCV} J. I. D\'{\i}az, D. G\'{o}mez-Castro, and J. L. V\'{a}zquez:
The fractional Schr\"{o}dinger equation with general nonnegative potentials.
The weighted space approach, Nonlinear Analysis,
2018, DOI: 10.1016/j.na.2018.05.001, arXiv:1804.08398.

\bibitem{Duvaut-Lions} G. Duvaut, J.-L. Lions, \emph{Les Inéquations en Mechanique et en Physique}. Dunod, Paris, 1972.

\bibitem{GanNeussKnabner} M. Gahn, M. Neuss-Radu and P. Knabner,
Homogenization of reaction-diffusion processes in a two-component porous
medium with nonlinear flux conditions at the interface, SIAM Journal on
Applied Mathematics (2016) 76 (5) 1819-1843.

\bibitem{Ghos} A. K. Ghosh, B.-H. Jeong, X. Huang, E. M.V. Hoek, Impacts of
reaction and curing conditions on polyamide composite reverse osmosis
membrane properties, Journal of Membrane Science 311 (2008) 34--45.

\bibitem{GoLoPePoSh17} D. Gomez, M.Lobo, E. Perez, A.V. Podol'skii, T.A.
Shaposhnikova. Unilateral problems for the p-Laplace operator in perforated
media involving large parameters. ESAIM Control Optim. Calc. Var. 2017.
DOI:10.1051/cocv/2017026.

\bibitem{GoLoPeSh11} D. Gomez, M. Lobo, E. Perez, T.A. Shaposhnikova.
Averaging in variational inequalities with nonlinear restrictions along
manifolds. CR Mecanique 339 (2011) 406-410.

\bibitem{GoLoPeSh13} D. Gomez, M. Lobo, E.Perez, T.A. Shaposhnikova.
Averaging of variational inequalities for the Laplacian with nonlinear
restrictions along manifolds. Appl. Anal. 92 (2013) 218-237.

\bibitem{Gon97} M. V. Goncharenko. Asymptotic behavior of the third
boundary-value problem in domains with fine-grained boundaries. Gakuto,
1997, p. 203-213.

\bibitem{JaNeSh} W. Jager, M. Neuss-Radu, T.A. Shaposhnikova. Homogenization
of a variational inequality for the Laplace operator with nonlinear
restriction for the flux on the interior boundary of a perforated domain.
Nonlinear Analysis: Real World Applications 15(2014), pp. 367-380.

\bibitem{Jamal} K. Jamal, M.A. Khan and M. Kamil, Mathematical modeling of
reverse osmosis sytems, Desalination, 160 (2004) 29-42

\bibitem{Kaizu89} S. Kaizu, The Poisson equation with semilinear boundary
conditions in domains with many tiny holes, J. Fac. Sci. Univ. Tokyo Sect.
IA Math. 36 (1989), p. 43-86.

\bibitem{Kaizu91} S. Kaizu. The Poisson equation with nonautonomous
semilinear boundary conditions in domains with many tiny holes. SIAM
J.Math. Anal. 1991, V. 22, N. 5, pp. 1222-1245.

\bibitem{LoOlPeSh} M. Lobo, O.A. Oleinik , M.E. Perez and T.A.
Shaposhnikova, On homogenization of solutions of boundary value problems in
domains perforated along manifolds, Ann. Scuola Norm. Super. Pisa Cl. Sci.
Ser. 4. \ \textbf{25} (1997) 611-629.

\bibitem{LoPesuSha 2011} M.Lobo, M.E. Perez, V.V. Sukhrev, T.A.
Shaposhnikova: Averaging of the boundary value problem in domain perforated
along (n-1)-dimensional manifold with nonlinear third type boundary
condition on the boundary of cavities, Doklady Math., 83 (2011) 34-38.

\bibitem{Mohammad} A. Mohammad, Y. Teow, W. Ang, Y. Chung, D. Oatley-Radcli
and N. Hilal, Nanofiltration membranes review: Recent advances and future
prospects, Desalination 356 (2015) 226-254.

\bibitem{OlSh1995} Oleinik O.A., Shaposhnikova T. A. On homogenization
problem for the Laplace operator in partially perforated domains with
Neumann conditions on the boundary of cavities. Rend. Mat. Acc. Lincei.
1995. S. 9, V. 6. P. 133-142.

\bibitem{OlSh1996} O.A. Oleinik, T.A. Shaposhnikova On the homogenization of
the Poisson equation in partially perforated domains with arbitrary density
of cavities and mixed type conditions on their boundary.
Rend.Mat.Accad.Lincei, s.9, v.7 (1996), p.129-146.

\bibitem{Shaposhnikova+Oleinik:1995Lincei} O.A. Oleinik, T.A. Shaposhnikova: On homogenization problems for the Laplace operator in partially
perforated domains with Neumann condition on the boundary of cavities. Atti
della Accad. Naz. dei Lincei. Cl. di Sci. Fis. Mat. e Nat. Rend. Lincei.
Mat. e Appl. 6, (1995) 133-142.

\bibitem{PeSha} M.E. Perez, T.A. Shaposhnikova: Boundary Homogenization of
Variational inequality with nonlinear restrictions for the flux on small
regions lying on the part of the boundary. Doklady Math. 85 (2012) 1-6.

\bibitem{PeShZu14} M.E. Perez, M.N. Zubova, T.A. Shaposhnikova.
Homogenization problem in a domain perforated by tiny isoperimetric holes
with nonlinear Robin type boundary conditions, Doklady Mathematics, 90
(2014) 489-494.



\bibitem{ZuSh07} M.N. Zubova, T.A. Shaposhnikova.  Homogenization of the
variational inequality corresponding to a problem with rapidly varying
boundary conditions, Math. Notes, 82 4 (2007) \ 481-491.

\bibitem{ZuShDiffEq} M.N. Zubova, T.A. Shaposhnikova. Homogenization of
boundary value problems in perforated domains with the third boundary
condition and the resulting change in the character of the nonlinearity in
the problem, Different. Equats. 47 (2011) 78-90.
\end{thebibliography}

\end{document}